\documentclass{amsart}

\usepackage{amssymb,euscript,amsmath, mathrsfs}
\usepackage[dvips]{graphicx}
\usepackage[dvips]{color}
\usepackage{epsfig}
\usepackage{MnSymbol}

\newcounter{ENUM}
\newcommand{\itm}{\item}
\newenvironment{ilist}{\renewcommand{\theENUM}{\roman{ENUM}}\renewcommand{\itm}{\addtocounter{ENUM}{1}\item[(\theENUM)]}\begin{itemize}\setcounter{ENUM}{0}}{\end{itemize}}
\newenvironment{alist}[1][0]{\renewcommand{\theENUM}{\alph{ENUM}}\renewcommand{\itm}{\addtocounter{ENUM}{1}\item[\theENUM)]}\begin{itemize}\setcounter{ENUM}{#1}}{\end{itemize}}

\def\Lie{{\mathscr Lie}}
\def\Eil{{\mathscr Eil}}
\def\LC{{\bf LC}}
\def\brkone{[ \cdot , \cdot]}
\def\brktwo{\langle \cdot , \cdot \rangle}
\def\brk{\{ \cdot , \cdot \}}

\def\OG{{\mathcal OG}}
\def\BT{{\mathcal BT}}
\def\QBT{{\mathcal QBT}}
\def\cO{{\mathcal O}}
\def\cM{{\mathcal M}}
\def\sG{{\mathscr G}}
\def\sB{{\mathscr B}}
\def\sO{{\mathscr O}}
\def\Com{{\mathscr Com}}
\def\gr{\operatorname{gr}}
\def\rlex{\operatorname{rlex}}
\def\op{\operatorname{op}}
\def\ind{\operatorname{ind}}

\def\sP{{\mathscr P}}
\def\sQ{{\mathscr Q}}
\def\cR{{\mathcal R}}

\def\cG{{\mathcal G}}

\def\cB{{\mathcal B}}

\newtheorem{thm}{Theorem}[section]
\newtheorem{prop}[thm]{Proposition}
\newtheorem{lem}[thm]{Lemma}
\newtheorem{cor}[thm]{Corollary}

\theoremstyle{definition}
\newtheorem{defn}[thm]{Definition}
\newtheorem{ques}[thm]{Question}
\newtheorem{ex}[thm]{Example}

\theoremstyle{remark}

\newtheorem{rem}[thm]{Remark}

\numberwithin{equation}{section}

\begin{document}
\title[Combinatorial bases for $\Lie_2(n)$ and $\sP_2(n)$]{Combinatorial bases for multilinear parts of free algebras with double compatible brackets}
\author{Fu Liu}
\begin{abstract}
Let $X$ be an ordered alphabet. $\Lie_2(n)$ (and $\sP_2(n)$ respectively) are the multilinear parts of the free Lie algebra (and the free Poisson algebra respectively) on $X$ with a pair of compatible Lie brackets. In this paper, we prove the dimension formulas for these two algebras conjectured by B. Feigin by constructing bases for $\Lie_2(n)$ (and $\sP_2(n)$) from combinatorial objects. We also define a complementary space $\Eil_2(n)$ to $\Lie_2(n),$ give a pairing between $\Lie_2(n)$ and $\Eil_2(n)$, and show that the pairing is perfect.

\end{abstract}
\maketitle

\section{Introduction}
Fix a commutative ring $R$ with unit. We recall a {\it Lie algebra} over $R$ is an $R$-module $V$ equipped with a bilinear binary operation  $\brkone,$ called a {\it Lie bracket}, satisfying two properties: for any $x,y,z \in V,$
\begin{eqnarray*}
&\mbox{antisymmetry}& [x, y] = - [y,x],  \\
&\mbox{Jacobi identity}&  [x , [y, z]] + [y, [z, x]] + [z, [x, y]] = 0.
\end{eqnarray*}
A closely related type of algebra is the Poisson algebra. A {\it Poisson algebra} over $R$ is an $R$-module $V$ equipped with two bilinear binary operations: a Lie bracket $\brkone$ and an associative commutative multiplication such that the Lie bracket is a {\it derivation} of the commutative multiplication: that is, for any $x, y, z \in V,$ we have
\begin{equation*}
[x, yz] = [x, y]z + y[x, z].
\end{equation*}

Let $X = \{x_1 < x_2 < \dots < x_n \}$ be an ordered alphabet. The {\it free Lie algebra} on $X$ over $R$ is the Lie algebra over $R$ that is generated by all possible Lie bracketings of elements of $X$ with no relations other than antisymmetries and Jacobi identities. Let $\Lie(n)$ be the {\it multilinear part} of this free Lie algebra: i.e., the subspace consisting of all elements containing each $x_i$ exactly once. We define the free Poisson algebra on $X$ similarly, and let $\sP(n)$ be its multilinear part. It is well known that the rank of $\Lie(n)$ is $(n-1)!$ and $\sP(n)$ is $n!.$ 

In this paper, we consider a free algebra on $X$ with two Lie brackets $\brkone$ and $\brktwo,$ which are {\it compatible}: that is, any linear combination of them is a Lie bracket. In fact, if we write out this condition explicitly, the compatibility gives one condition in addition to the antisymmetry and Jacoby identity for each of the two brackets. We call this additional condition the {\it mixed Jacobi identity}. For easy reference, we put these five relations together here: for any $x, y, z,$
\begin{itemize}
\item[(S1)] $[x, y] + [y,x] = 0,$
\item[(S2)] $\langle x, y \rangle + \langle y , x \rangle = 0,$
\item[(J1)] $ [x , [y, z]] + [y, [z, x]] + [z, [x, y]] = 0,$
\item[(J2)] $ \langle x , \langle y, z \rangle \rangle + \langle y, \langle z, x\rangle\rangle+ \langle z, \langle x, y\rangle\rangle = 0,$
\item[(MJ)] $ [x , \langle y, z\rangle ] + [y, \langle z, x\rangle ] + [z, \langle x, y\rangle] + \langle x , [y, z]\rangle + \langle y, [z, x]\rangle + \langle z, [x, y]\rangle = 0.$
\end{itemize}

Let $\Lie_2(n)$ be the multilinear part of this free algebra. Similarly, we let $\sP_2(n)$ be the multilinear part of the free algebra with two compatible Lie brackets and one associative commutative multiplication, where both of the Lie brackets are derivations of the commutative multiplication. Therefore, in addition to (S1), (S2), (J1), (J2) and (MJ), there are two more kinds of relations in $\sP_2(n):$ for any $x, y, z,$
\begin{itemize}
\item[(D1)] $[x, yz] - [x, y]z - y[x, z] = 0,$
\item[(D2)] $\langle x, yz \rangle - \langle x, y \rangle z - y\langle x, z  \rangle= 0,$
\end{itemize}

Several years ago, B. Feigin conjectured that these spaces may be connected with the work of M. Haiman. As a result, Feigin gave conjectural formulas for the ranks of $\Lie_2(n)$ and $\sP(n),$ which are two of the main theorems of this paper. 

\begin{thm}\label{main1}
$\Lie_2(n)$ is free of rank $n^{n-1}$.
\end{thm}
\begin{thm}\label{main2}
$\sP_2(n)$ is free of rank $(n+1)^{n-1}.$
\end{thm}

It turns out these two theorems are equivalent to each other (Corollary \ref{equiv}). Therefore, it is enough to show one of the theorems. We will focus on Theorem \ref{main1} in this paper. The basic idea is to construct a set of $n^{n-1}$ monomials which spans $\Lie_2(n)$, and to prove linear independence via the use of a perfect pairing.

The plan of this paper is as follows: In sections 2 and 3, we define basic combinatorial objects $\overline{\cG}_n,$  construct a set of monomials $\cB_n(X)$ from $\overline{\cG}_n$ and show $\cB_n(X)$ spans $\Lie_2(n).$ Section 4 is independent from the rest of the paper. It gives a purely algebraic way to show the independence of $\cB_n(X),$ and thus conclude Theorem \ref{main1}. In sections 5--7, we give another approach to proving the independence of $\cB_n(X).$ We introduces new combinatorial objects to describe $\Lie_2(n)$ as well as another space $\Eil_2(n),$ and define a pairing between these two spaces. Then by showing this pairing is a perfect pairing (Theorem \ref{main3}), we conclude $\cB_n(X)$ is a basis for $\Lie_2(n)$ and Theorem \ref{main1}. In section 8, we give a sufficient condition for a set of monomials of $\Lie_2(n)$ to be a basis (Theorem \ref{main4}), which provides us more bases for $\Lie_2(n).$ Based on the relation between $\Lie_2(n)$ and $\sP_2(n),$ in section 9, we build bases for $\sP_2(n)$ from bases for $\Lie_2(n)$ (Proposition \ref{com}), which we show in section 10 are indeed bases for $\sP_2(n).$ The equivalence between Theorem  \ref{main1} and Theorem \ref{main2} is an immediate corollary to this result, and then we can complete the proof of Theorem \ref{main2}. We complete our paper with Section \ref{questions}, where we propose some possible direction for future research.

Finally, we mention that Dotsenko and Khoroshkin independently prove Theorems \ref{main1} and \ref{main2} in \cite{dot-kho} using the theory of operads. They also obtain character formulas for the representation of the symmetric groups and the $SL_2$ group in $\Lie_2(n)$ and $\sP_2(n).$ The approach in our paper is quite different from \cite{dot-kho}. Our method is more combinatorial, and we create more bases for $\Lie_2(n).$ We expect that our additional bases will have applications to the theory of operads. 


\section{Two-Colored Graphs and Rooted trees}
Our chosen alphabet $X = \{x_1 < x_2 < \dots < x_n \}$ will form the vertex set of the combinatorial objects we are going to define.

\begin{defn}
A {\it two-colored graph} is a connected graph whose edges are colored by two colors, red and blue. We denote by $\cG_n$ the set of all two-colored graphs on $X.$
\end{defn}

\begin{defn}
A {\it tree} is a connected acyclic graph. A {\it rooted tree} is a tree with one special vertex, which we call it the {\it root} of the tree. (Note the edges of rooted trees here are not colored.) Let $\cR_n$ be
the set of all rooted trees on $X.$ 

For any edge $\{i,j\}$ in a rooted tree, if $i$ is closer to the root than $j,$ then we call $i$ the {\it parent} of $j$ and $j$ a {\it child} of $i.$ (It is clear that any non-root vertex has a unique parent, but can have multiple children.) Furthermore, if $i$ is the parent of $j$, we call the edge $\{i,j\}$ an {\it increasing edge} if $i < j$ and a {\it decreasing edge} if $i > j$.
\end{defn}
It is well known that the
cardinality of $\cR_n$ is $|X|^{|X|-1}=n^{n-1}$ \cite[Proposition 5.3.2]{ec2}. 

We define a color map $c$ from $\cR_n$ to $\cG_n$ as follows. Given any tree $T
\in \cR_n,$ we color all of the increasing edges by red and all of
the decreasing edges by blue,  and denote the resulting two-colored
tree by $c(T)$ (by treating the root of $T$ as a regular vertex).

Before we discuss the color map, we define a special subset of $\cG_n.$

\begin{defn}\label{2color} 
Let $G$ be a two-colored graph,
\begin{ilist}
\itm If $\exists i < j < k,$ such that $\{i,k\}, \{j, k\}$ are both red edges in $T,$ we say $T$ has pattern $1r3r2$ (or $2r3r1$) where $``r"$ stands for a red edge.
\itm If $\exists i < j < k,$ such that $\{i, j\}, \{i, k\}$ are both blue edges in $T,$ we say $T$ has pattern $2b1b3$ (or $3b1b2$), where $``b"$ stands for a blue edge.
\itm If $\exists i < j < k,$ such that $\{i, j\}$ is a red edge and $\{j, k\}$ is a blue edge in $T,$ we say $T$ has pattern $1r2b3$ (or $3b2r1$).
\end{ilist}

Let $\overline{\cG}_n$ be the set of all two-colored trees in $\cG_n$ avoiding patterns $1r3r2$, $2b1b3$, and $1r2b3.$
\end{defn}

\begin{lem}
The color map $c$ gives a bijection between $\cR_n$ and $\overline{\cG}_n.$ Hence, the cardinality of $\overline{\cG}_n$ is $n^{n-1}.$
\end{lem}

\begin{proof}
Given any rooted tree $T$ with root $r,$ suppose pattern $1r3r2$ appears in $c(T),$ which means $\exists i < j < k,$ such that $\{i,k\}, \{j, k\}$ are both red edges in $c(T).$ That $\{i, k\}$ is a red edge in $c(T)$ implies that it is an increasing edge in $T,$ and thus $i$ is the parent of $k.$ However, we similarly see that $j$ is the parent of $k.$ This contradicts the fact that $k$ can only have a unique parent. Hence, $c(T)$ does not have pattern $1r3r2.$ By similar arguments, we can exclude patterns $2b1b3$ and $1r2b3$. Therefore, $c(T)$ is a two colored tree in $\overline{\cG}_n,$ for any $T \in \cR_n.$ Thus, we can consider our color map to be a map from $\cR_n$ to $\overline{\cG}_n.$
 
Conversely, for any two-colored tree $G \in \overline{\cG}_n,$ we construct an oriented graph $G'$ on $X$ from $G$ as follows: for each edge $e=\{i,j\}$ in $G$ with $i < j,$ we point $i$ to $j$ in $G'$ if $e$ is red in $G$, and point $j$ to $i$ in $G'$ if $e$ is blue in $G$. Then the condition that $G$ avoids the patterns $1r3r2,$ $2b1b3,$ and $1r2b3$ implies that each vertex in $G'$ can have at most one edge pointing towards it. Because $G$ is acyclic, there is a unique vertex $r$ in $G'$ without edges pointing towards it. Therefore,  we can recover a rooted tree in $\cR_n$ from $G$ by choosing $r$ to be the root and forgetting the colors.

\end{proof}

\begin{rem}
Because of this bijection between $\cR_n$ and $\overline{\cG}_n$, in the rest of the paper, we will always consider these two sets to be the same set. In other words, when we talk about a rooted tree or an acyclic graph  $G \in \cR_n = \overline{\cG}_n,$ we consider it is a two-colored rooted tree such that
\begin{alist}
\itm G avoids patterns $1r3r2,$ $2b1b3,$ and $1r2b3$;
\itm each red edge is an increasing edge, and each blue edge is a decreasing edge.
\end{alist}
\end{rem}

\begin{ex}
In Figure \ref{exgr}, we show an example of how the color map gives a bijection between $\cR_n$ and $\overline{\cG}_n.$ The tree on the left side is a rooted tree in $\cR_3.$ We circle $x_3$ to indicate it is the root. Under the color map $c,$ we map the tree to the graph on the right side, which is in $\overline{\cG}_3.$ The graph in the middle is the two-colored rooted tree in $\overline{\cG}_3 = \cR_n$ we will consider from now on. In Figure \ref{exgr}, we still include the circle to indicate $x_3$ is the root, but in the Figures of the rest of the paper,  we will always draw the root on the highest level of a rooted tree to indicate the root instead of drawing circles.
\begin{figure}
\begin{center}
 
\input{exgr.pstex_t}
 
\caption{Examples of the bijection between $\cR_n$ and $\overline{\cG}_n.$}
\label{exgr}
\end{center}
\end{figure}
\end{ex}

\section{A basis candidate for $\Lie_2(n)$}
We will give a set of monomials of $\Lie_2(n)$ constructed from $\overline{\cG}_n =\cR_n,$ and show it spans $\Lie_2(n).$ We denote by $M_n$ the set of all monomials of $\Lie_2(n).$

\subsection{The construction of $\cB_n(X)$}
\begin{defn}\label{defBn}
For any graph $G$ in $\overline{\cG}_n = \cR_n$ with root $r,$ we define a monomial $b_G \in M_n$ recursively as follows:
\begin{ilist}
\itm If $G = r,$ let $b_G := r.$
\itm If $G \neq  r,$ let $c_1 < \cdots < c_k$ be the vertices connected to $r,$ and $G_1, \dots, G_k$ be the corresponding subtrees. 
\begin{itemize}
\item If $r < c_k$, i.e., there are red edges adjacent to $r,$ choose the smallest $c_i$ such that $\{r, c_i\}$ is a red edge. Let $b_G := [b_{G\setminus G_i}, b_{G_i}]$. 
\item If $r > c_k,$ i.e., all the edges adjacent to $r$ are blue, let $b_G := \langle b_{G_k}, b_{G \setminus G_k} \rangle.$ 
\end{itemize}
\end{ilist}
We define $\cB_n(X)$ to be the set of all monomials obtained from $\overline{\cG}_n=\cR_n:$
$$\cB_n(X) := \{ b_G \ | \ G \in \overline{\cG}_n \}.$$

\end{defn}

\begin{rem}
From the way we construct $b_G,$ it is clear that each red edge (or increasing edge) becomes $\brkone$, and each blue edge (or decreasing edge) becomes $\brktwo.$ 
\end{rem}

One checks that different trees give different monomials. Thus, the cardinality of $\cB_n(X)$ is $n^{n-1}$ as well.

\begin{ex}
When $n = 3,$ there are $3^{3-1} = 9$ rooted trees in $\overline{\cG}_n=\cR_n.$ In Figure \ref{exbn}, we show these 9 graphs together with their corresponding $b_G$'s defined in Definition \ref{defBn}. The $9$ monomials $b_G$'s shown in Figure \ref{exbn} are the elements in $\cB_3(X).$

\begin{figure}
\begin{center}
 
\input{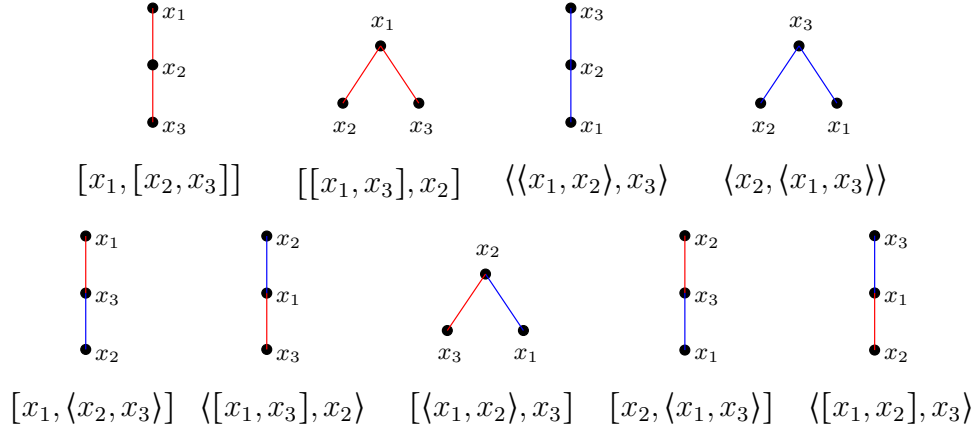}
 
\caption{Examples of the construction of $\cB_n(X).$}
\label{exbn}
\end{center}
\end{figure}

\end{ex}

We need to discuss properties of $\cB_n(X)$ before showing it spans $\Lie_2(n).$ 

\begin{defn}\label{root}
For any monomial $m \in M_n,$ we define the {\it graphical root} of $m$ recursively:
\begin{ilist}
\itm If $m = x,$ a single variable, let $\gr(m) := x;$
\itm If $m = [m_1, m_2],$ let $\gr(m) := \min \{\gr(m_1), \gr(m_2) \};$
\itm If $m = \langle m_1, m_2\rangle,$ let $\gr(m) := \max \{\gr(m_1), \gr(m_2) \};$
\end{ilist}
\end{defn}
It is clear from our definition that for $G \in \overline{\cG}_n = \cR_n,$ the graphical root of the monomial $b_G$ is exactly the root of $G.$ Using this definition, we are able to give an equivalent definition of $\cB_n(X).$

\begin{lem}\label{crtbss}
$\cB_n(X)$ is the set of all monomials $m$ in $\Lie_2(n)$ satisfying:
\begin{alist}
\itm If $n=1$ and $X = \{x\},$ then $m = x.$
\itm If $m = \{m_1, m_2\},$ where $\{ \cdot , \cdot \} = \brkone$ or $\brktwo,$ and suppose $X_i$ is the set of letters in $m_i,$ for $i=1,2,$ then $\gr(m_1) < \gr(m_2)$ and $m_i \in \cB_{|X_i|}(X_i),$ for $i = 1,2.$
\itm If $m = [[m_1, m_3], m_2],$ then $\gr(m_2) < \gr(m_3).$
\itm If $m = \langle m_1, m_2 \rangle,$ then $m_2$ is a letter or has the form $\brktwo.$
\itm If $m = \langle m_2, \langle m_1, m_3\rangle \rangle,$ then $\gr(m_1) < \gr(m_2).$
\end{alist}
\end{lem}

\begin{proof}
One checks that the conditions b), c), d) and e) correspond to the recursive step (ii) in Definition \ref{defBn}. Using this, the lemma can be checked by induction on $n.$

%

\end{proof}

\subsection{$\cB_n(X)$ spans $\Lie_2(n)$}

We define an algorithm recursively that takes a monomial $m \in M_n$ as input, and expresses $m$ as a linear combination of monomials in $\cB_n(X).$ We show the algorithm below first, and then prove the algorithm will terminate on any monomial $m \in M_n.$

\subsection*{Algorithm LC}
\begin{enumerate}
\itm If $m$ is in $\cB_n(X),$ then  output $\LC(m) = m.$ Otherwise, $m \not\in \cB_n(X),$ and then $m$ must have the form $\{m_1, m_2\},$ where $\brk = \brkone$ or $\brktwo.$ Suppose $X_i$ is the set of letters in $m_i,$ for $i=1,2.$
\itm If $m_1 \not\in \cB_{|X_1|}(X_1)$ or $m_2 \not\in \cB_{|X_2|}(X_2),$ then run $\LC$ on $m_1$ and $m_2.$ Suppose we get $$\LC(m_1) = \sum_{b_{i} \in \cB_{|X_1|}(X_1)} \alpha_i b_i, \ \ \ \LC(m_2) = \sum_{b_{j}' \in \cB_{|X_2|}(X_2)} \beta_j b_j'.$$ Output
$$\LC(m) = \sum_{i,j} \alpha_i \beta_j \LC(\{b_i, b_j'\}).$$
\itm If $\gr(m_1) > \gr(m_2),$ output $\LC(m) = - \LC(\{m_2, m_1\}).$ 
\itm If the algorithm reaches this step. we must have $m \not\in \cB_n(X)$ and  $m$ satisfies condition b) in Lemma \ref{crtbss}. There are two more situations we need to deal with.
\begin{ilist}
\itm If $\brk = \brkone,$ then $m$ does not satisfy Lemma \ref{crtbss}/condition c). Hence $m_1 = [m_1', m_1''],$ so $m = [[m_1', m_1''], m_2]$ where $\gr(m_1'') < \gr(m_2).$ Note $m_1 \in \cB_{|X_1|}(X_1),$ so $\gr(m_1') < \gr(m_1'') < \gr(m_2).$ Output $$\LC(m) = \LC([[m_1', m_2], m_1'']) + \LC([m_1', [m_1'', m_2]])$$
\itm If $\brk = \brktwo,$ then $m$ does not satisfy either condition d) or c) of  Lemma \ref{crtbss}. Thus, $m_2$ has the form $[m_2', m_2'']$ or $\langle m_2', m_2''\rangle.$ In either case, one can conclude that $\gr(m_1) < \gr(m_2') < \gr(m_2'').$
 \begin{enumerate}
 \item If $m_2 = [m_2', m_2'']$, then $m = \langle m_1, [m_2', m_2'']\rangle.$ Output
\begin{multline*}
\LC(m) = - \LC(\langle [ m_1, m_2''], m_2' \rangle )+ \LC(\langle [m_1, m_2'], m_2''\rangle)  \\
 -\LC([m_1, \langle m_2', m_2'' \rangle] ) + \LC([m_2', \langle m_1, m_2'' \rangle ])+ \LC([\langle m_1, m_2' \rangle, m_2'']).
\end{multline*}
\item If  $m_2 = \langle m_2', m_2'' \rangle,$ then $m = \langle m_1, \langle m_2', m_2'' \rangle \rangle.$ Output $$\LC(m) =\LC( \langle m_2', \langle m_1, m_2'' \rangle \rangle) + \LC(\langle \langle m_1, m_2' \rangle, m_2'' \rangle ).$$
\end{enumerate}
\end{ilist}
\end{enumerate}

\begin{lem}\label{term}
Suppose $m$ is a monomial in $\Lie_2(n)$ satisfying condition b) in Lemma \ref{crtbss}, i.e., $m = \{m_1, m_2\}$ with $\gr(m_1) < \gr(m_2)$ and $m_i \in \cB_{|X_i|}(X_i),$ for $i = 1,2,$ where $\{ \cdot , \cdot \} = \brkone$ or $\brktwo,$ and $X_i$ is the set of letters in $m_i,$ for $i=1,2.$ Then we have the following results.

\begin{enumerate}
\itm If $\brk = \brkone,$ then $\LC(m)$ terminates. Furthermore, for any monomial $b \in \cB_n(X)$ appearing in $\LC(m)$ with nonzero coefficient,  we have $\gr(b) = \gr(m)$ and the outermost bracket of $b$ is $\brkone.$
\itm If $\brk = \brktwo,$ then $\LC(m)$ terminates. Furthermore, for any  monomial $b \in \cB_n(X)$ appearing in $\LC(m)$ with nonzero coefficient,  if the outermost bracket of $b$ is $\brktwo$ then $\gr(b) \ge \gr(m).$
\end{enumerate}
\end{lem}

\begin{proof}
We prove the lemma by induction on $n.$ The base case $n=1$ is clearly true. Assume the lemma is true when the size of the alphabet smaller than $n.$ Under this assumption, we prove separately that (1) and (2) hold when $|X| = n$.

\begin{itemize}
\itm If $m = [m_1, m_2],$ we prove (1) by induction on $\gr(m_2)$.
The smallest possibility for $\gr(m_2)$ is $x_2$ in which case $\gr(m_1) = x_1.$ Hence, if $m = [[m_1', m_1''], m_2],$ we must have $\gr(m_1'') > x_2 = \gr(m_2).$ Thus, $m \in \cB_n(X)$ and $\LC(m) = m,$ so (1) holds. Now we assume (1) holds for the cases where $\gr(m_2) < x_k.$ Suppose $\gr(m_2) = x_k.$ It is clear we only need to consider the case that $m \not\in \cB_n(X).$ In this case, $m$ has the form $[[m_1', m_1''], m_2],$ where $\gr(m_1') < \gr(m_1'') < \gr(m_2).$ (Note since $m_1 = [m_1', m_1'']$ is in $\cB_n(X),$ we must have that $m_1' \in \cB(X_1')$ and $m_1'' \in \cB(X_1''),$ where $X_1'$ and $X_1''$ are the corresponding alphabets.) By how the algorithm $\LC$ is designed, we will use the formula in step (4)/(i), so it suffices to show that both $\LC([[m_1', m_2], m_1''])$ and $\LC([m_1', [m_1'', m_2]])$ terminate, and for each monomial $b$ appearing in the output, $b$ satisfies $\gr(b) = \gr(m)=\gr(m_1')$ and the outermost bracket of $b$ is $\brkone.$  

We only show it for $\LC([[m_1', m_2], m_1''])$, since a similar argument applies to $\LC([m_1', [m_1'', m_2]]).$ Since the size of the supporting alphabet of $[m_1', m_2]$ is smaller than $n,$ we can apply the induction hypothesis to get $\LC([m_1', m_2]) = \sum_{i} \alpha_i b_i,$ where all the $b_i$ appearing in the linear combination satisfy $\gr(b_i) = \gr([m_1', m_2]) = \gr(m_1').$
Using step (2) in the algorithm, $\LC([[m_1', m_2], m_1'']) = \sum_i \alpha_i \LC([b_i, m_1'']).$ 
For each $b_i,$ $\gr(b_i) = \gr(m_1') < \gr(m_1''),$ so $\gr([b_i, m_1'']) = \gr(m_1').$ Hence $[b_i, m_1'']$ satisfies the hypothesis in the Lemma with $\gr(m_1'') < \gr(m_2) = x_k.$ By the induction hypothesis, $ \LC([b_i, m_1'']),$ and therefore $\LC([[m_1', m_2], m_1'']),$ terminates and outputs a linear combination of  monomials satisfying the desired properties.

\itm If $m = \langle m_1, m_2 \rangle,$ we prove (2) by reverse induction on $\gr(m_1).$ The highest possible value for $\gr(m_1)$ is $x_{n-1},$ in which case $\gr(m_2) = x_n.$ Similarly to the case when $m=[m_1, m_2],$ one shows $m \in \cB_n(X)$ and so $\LC(m) =m.$ Thus, (2) holds. Now assume (2) holds whenever $\gr(m_2) > x_k.$ Suppose $\gr(m_2) = x_k.$ We discuss two possibilities for $m_2.$
\begin{ilist}
\itm If $m_2 = \langle m_2', m_2'' \rangle,$ we can show (2) holds by similar arguments to those we used for the cases where $m =[m_1, m_2].$
\itm If $m_2$ is a single letter or $m_2 = [m_2', m_2''],$ we use another level of induction on the size of $X_2,$ the supporting alphabet of $m_2.$ If $|X_2| = 1,$ then $m_2$ is a single letter. Thus, $m \in \cB_n(X)$ and $\LC(m) = m,$ so (2) holds. Assume (2) holds whenever $|X_2| < \ell$. Suppose $|X_2| = \ell (>1);$ then $m_2$ has the form $[m_2', m_2''].$ We have $\gr(m_1) < \gr(m_2) = \gr(m_2') < \gr(m_2'').$ For this case, we will use the formula in step(4)/(ii)/(a), so it remains to show that  $\LC(\langle [ m_1, m_2''], m_2' \rangle )$, $\LC(\langle [m_1, m_2'], m_2''\rangle)$,
$\LC([m_1, \langle m_2', m_2'' \rangle] ),$ $\LC([m_2', \langle m_1, m_2'' \rangle ])$ and $\LC([\langle m_1, m_2' \rangle, m_2''])$ all terminate with output having the desired properties.
However, the outermost brackets of the last three monomials are $\brkone.$ Thus, using step (2) of the algorithm, they all become linear combinations of monomials of the form $[m_3, m_4],$ which we have already showed will terminate under $\LC,$ and all the monomials appearing in the output satisfy the desired properties in (1), which leads to the desired properties in (2). Therefore, we only need to check $\LC(\langle [ m_1, m_2''], m_2' \rangle )$ and $\LC(\langle [m_1, m_2'], m_2''\rangle).$ 

We only check $\LC(\langle [m_1, m_2'], m_2''\rangle);$ similar arguments would hold for $\LC(\langle [ m_1, m_2''], m_2' \rangle ).$ First, we have $$\gr (\langle [m_1, m_2'], m_2''\rangle) = \gr(m_2'') \ge \gr(m_2' ) = \gr(m).$$ If $m_2''$ has the form $\langle m_3, m_4\rangle,$ i.e., case (i), we have shown $\LC$ will terminate and for each monomial $b$ in the output, if the outermost bracket is $\brktwo,$ then $\gr(b) \ge \gr(\langle [m_1, m_2'], m_2''\rangle) \ge \gr(m).$ If $m_2''$ is a letter or $m_2'' = [m_3, m_4],$ then since the supporting alphabet of $m_2''$ is a proper subset of $X_2,$ it is strictly smaller than $|X_2| = \ell.$ Therefore, we can apply the induction hypothesis to get the desired result.
\end{ilist}

\end{itemize}
\end{proof}

\begin{lem}

For any monomial $m \in \Lie_2(n),$ $\LC(m)$ terminates and the output expresses $m$ as a linear combination of monomials in $\cB_n(X).$
\end{lem}

\begin{proof}
First, as we discussed inside the algorithm, the steps (1)-(4) include all possible situations. Hence, it won't happen that the algorithm becomes stuck without giving output. 
Then, the statement of termination can be proved by induction on $n$ and using Lemma \ref{term}. 

It is left to show the linear combination of monomials $\LC(m)$ output is equal to $m.$ This can be done by checking all the formulas involved agree with the properties of the operations: bilinearity, antisymmetry, Jacobi identity and the compatibility of two brackets. 

\end{proof}

We have shown that our algorithm $\LC$ works. Hence, we conclude:
\begin{prop}\label{liespan}
$\cB_n(X)$ spans $\Lie_2(n).$
\end{prop}

\section{The first proof of Theorem \ref{main1}}

In this section, we will use the following two lemmas suggested by Brian Osserman to prove $\Lie_2(n)$ is isomorphic to a free $R$-module of rank $n^{n-1},$ thus conclude Theorem \ref{main1}. Because we will give another proof of Theorem \ref{main1} in later sections, we only present the idea and give a partial proof. The reason we include this section is that the idea we use here does not require defining or using new objects and we believe it is easier to apply this idea more generally in similar situations, when one wants to prove a basis candidate is indeed a basis provided we know a way to write any element in the module as a linear combination of the elements in the basis candidate.

\begin{lem}\label{bbo1}
Let $U, V, W$ be three $R$-modules. Suppose there exist homomorphisms $f: U \to V,$ $g: U \to W$ and $h: W \to V$ satisfying the following conditions.
\begin{ilist}
\itm $f, g$ and $h$ are surjective.
\itm $f =  h \circ g.$
\itm $\ker(f) \subset \ker(g).$
\end{ilist}
Then $V \cong W.$
\end{lem}

\begin{proof}
It is enough to show that $h$ is injective, i.e., $\ker(h) = 0.$ Suppose we have $x \in W,$ such that $h(x) = 0.$ Since $g$ is surjective, there exists $y \in U,$ such that $g(y) = x.$ Then we have $f(y) = h(g(y)) = h(x) = 0.$ Thus, $y \in \ker(f).$ Since $\ker(f) \subset \ker(g),$ we have $x= g(y) = 0.$ 
\end{proof}

This lemma gives one way to verify whether a set of elements spans a module is a basis.

\begin{lem}\label{bbo2}
Suppose $V$ is an $R$-module spanned by a set of elements $M = \{m_1, \dots, m_\ell\}$ with relations given by the set $Rel.$ Let $\cB=\{b_1, \dots, b_k\}$ be a subset of $M$ such that $\cB$ spans $V$. In particular, for any element $m_i$ of $M,$ it can be written as linear combination of $\cB.$
Although there might be multiple ways to write $m_i$, we fix one of them:  $$m_i = \sum_{j=1}^k \alpha_{i,j} b_j = (\alpha_{i,1}, \dots, \alpha_{i,k}) \cdot (b_1, \dots, b_k)^T.$$
 We write $$\alpha_i =  (\alpha_{i,1}, \dots, \alpha_{i,k}) \in R^k.$$
If for any relation $r$ in $Rel,$ we have
\begin{equation}\label{rel}
r \mbox{ is } \sum_j \gamma_j m_j = 0 \ \Longrightarrow \ \sum_j \gamma_j \alpha_j = 0,
\end{equation}
then $V$ is free of rank $k.$ Furthermore, $\cB$ is a basis for $V.$
\end{lem}

\begin{proof} 
We do the following setup.
\begin{itemize}
\itm Let $U$ be a free $R$-module of rank $\ell$ with a basis $\{u_1, \dots, u_\ell \}.$ 
\itm Let $W$ be a free $R$-module of rank $k$ with a basis $\{ w_1, \dots, w_k \}.$ 
\itm Let $f$ be the homomorphism from $U$ to $V$ obtained by mapping each $u_i$ to $m_i$ in $M.$
\itm Let $g$ be the homomorphism from $U$ to $W$ obtained by mapping each $u_i$ to $$g(u_i) = \alpha_i \cdot (w_1, \dots, w_k)^T = \sum_j \alpha_{i,j} w_j.$$
\itm Let $h$ be the homomorphism from $W$ to $V$ by mapping each $w_i$ to $b_i.$
\end{itemize}
If $f, g$ and $h$ satisfy all the conditions listed in Lemma \ref{bbo1}, then we have $V \cong W,$ which implies $rank(V) = rank(W) = k.$ Since $\cB$ has cardinality $k$ and spans $V,$ $\cB$ is a basis for $V.$ Therefore, our goal is to verify the three conditions in Lemma \ref{bbo1}.

It is clear that $f, g$ and $h$ are all surjections. For any $u_i,$ 
$$h(g(u_i)) = h( \sum_j \alpha_{i,j} w_j) =  \sum_j \alpha_{i,j} b_j = m_i = f(u_i).$$
Hence, $f = h \circ g.$ 
\begin{eqnarray*}
\sum_j \gamma_j u_j \in \ker(f) &\Rightarrow& \sum_j \gamma_j m_j = f(\sum_j \gamma_j u_j) =  0 \in Rel \\
&\Rightarrow& \sum_j \gamma_j \alpha_j = 0 \\
&\Rightarrow& g(\sum_j \gamma_j u_j) = \sum_j \gamma_j \alpha_j (w_1, \dots, w_k)^T = 0 \\
&\Rightarrow& \sum_j \gamma_j u_j \in \ker(g).
\end{eqnarray*}
Therefore, $\ker(f) \subset \ker(g).$

\end{proof}

\begin{rem}\label{bbo3}
Note that if we have a set of relations $\{r_i\}$ that satisfy \eqref{rel}, then any finite linear combinations of $r_i$'s satisfy \eqref{rel} as well. Therefore, If $Rel$ is an $R$-module generated by a set of relations $\{r_i\},$ such that for each $r \in \{r_i\},$ $r$ satisfies \eqref{rel}, then we can make the conclusions in Lemma \ref{bbo2}.
\end{rem}

We will use Lemma \ref{bbo2} (and this Remark) to prove Theorem \ref{main1}. We first describe the elements that generate the relation set for our problem.
\begin{defn}We say a relation is {\it of the form (S1)} if it can be written as 
$$[x, y] + [y,x] = 0, \mbox{or a multiple of it, e.g., } [z, \langle [x,y], w \rangle] + [z, \langle [y,x], w \rangle],$$
for some $x, y$ (and $z, w$).

Similarly, we define relations that are {\it of the forms (S2), (J1), (J2) and (MJ)}.
\end{defn}

In our problem, $V = \Lie_2(n)$ is an $R$-module spanned by the set of all monomials $M_n$ of $\Lie_2(n)$ with a set of relations $Rel_n,$ which are generated by relations that are of the forms (S1), (S2), (J1), (J2) and (MJ). $\cB_n(X)$ spans $\Lie_2(n).$ In particular, the algorithm $\LC$ gives one way to write any monomial $m \in M_n$ as a linear combinations of $b_G \in \cB_n(X).$ We naturally define $\alpha$ in terms of $\LC.$ For any $m \in M_n,$ if $\LC(m) = \sum_{i,T} \alpha_G b_G,$ let $$\alpha_m = (\alpha_{G} \ | \ G \in \overline{\cG}_n).$$
Note that we can fix an ordering of graphs in $\overline{\cG}_n = \cR_n,$ and consider $\alpha_m$ as a vector in $R^{n^{n-1}}.$

\begin{lem}\label{ind}
For any relation $\sum \gamma_m m = 0$ of monomials in $\Lie_2(n)$ of the form (S1), (S2), (J1), (J2), or (MJ). we have $\sum \gamma_m \alpha_m = 0.$
\end{lem}

Given this lemma, Lemma \ref{bbo2}, together with Remark \ref{bbo3}, implies Theorem \ref{main1} and that $\cB_n(X)$ is a basis for $\Lie_2(n).$

Hence, it is left to prove Lemma \ref{ind}. As we mentioned at the beginning of the section, we will only give a partial proof of Lemma \ref{ind}. We prove the case when the relation is of the form (S1). In fact, for all other cases, one can argue similarly but with more complicated arguments. 

\begin{proof}[(Incomplete) proof of Lemma \ref{ind}] 
We prove the lemma by induction on $n,$ the cardinality of the alphabet $X.$ The base case $n=1$ is trivial. Now assume the lemma is true when the size of the alphabet is smaller than $n (\ge 2).$ Suppose $|X| = n.$

If the relation $\sum \gamma_m m = 0$ is of the form (S1), then there are three possible cases:
\begin{ilist}
\itm $\{ m_1, m_2 \} + \{ m_1', m_2 \} = 0,$ where $m_1 + m_1' = 0$ is a relation of the form (S1) for corresponding alphabet $X_1.$ 
\itm $\{ m_1, m_2 \} + \{ m_1, m_2' \} = 0,$ where $m_2 + m_2' = 0$ is a relation of the form (S1) for corresponding alphabet $X_2.$
\itm $[m_1, m_2] + [m_2, m_1] = 0.$ 
\end{ilist}
In case (i), suppose 
$\LC(m_1) = \sum_{b_i \in \cB(X_1)} \alpha_i b_i, \LC(m_1'
) = \sum_{b_i \in \cB(X_1)} \alpha_i' b_i,$ and $\LC(m_2) = \sum_{b_i \in \cB(X_2)} \beta_j b_j.$
Hence, $\alpha_{m_1} = (\alpha_1, \dots, \alpha_{|\cB(X_1)|})$ and $\alpha_{m_1'} = (\alpha_1', \dots, \alpha_{|\cB(X_1)|}').$ By the induction hypothesis,
$$\alpha_{m_1} + \alpha_{m_1'} = 0 \Rightarrow \alpha_i + \alpha_i' = 0, \forall i.$$
According to step (2) of the algorithm $\LC,$ 
\begin{eqnarray*}
\LC(\{ m_1, m_2 \}) &=& \sum_{i,j} \alpha_i \beta_j \LC(\{b_i, b_j  \}), \\
\LC(\{ m_1', m_2 \}) &=& \sum_{i,j} \alpha_i' \beta_j \LC(\{b_i, b_j  \}) = -\sum_{i,j} \alpha_i \beta_j \LC(\{b_i, b_j  \}).
\end{eqnarray*}
Therefore, the coefficients in $\LC(\{ m_1, m_2 \})$ are exactly the negative of those in $\LC(\{ m_1', m_2 \}).$ Thus, $\alpha_{ \{m_1, m_2\} } + \alpha_{ \{m_1', m_2\} } = 0.$

Case (ii) can be showed similarly to case (i). 

In case (iii), suppose $\LC(m_1) = \sum_{b_i \in \cB(X_1)} \alpha_i b_i,$ and $\LC(m_2) = \sum_{b_i \in \cB(X_2)} \beta_j b_j.$ Again, according to step (2) of $\LC,$
\begin{eqnarray*}
\LC([m_1, m_2]) &=&  \sum_{i,j} \alpha_i \beta_j \LC([b_i, b_j  ]), \\
\LC([m_2, m_1]) &=&  \sum_{i,j} \alpha_i \beta_j \LC([b_j, b_i  ]).
\end{eqnarray*}
If we continue to step (3) of $\LC,$ depending on whether $\gr(b_i)$ is bigger or smaller than $\gr(b_j),$ one of $[b_i, b_j]$ and $[b_j, b_i]$ is going to be changed to the negative of the other. Therefore,  the coefficients in $\LC([m_1, m_2])$ are exactly the negative of those in $\LC([m_2, m_1]).$ Thus, $\alpha_{ [m_1, m_2] } + \alpha_{ [m_2, m_1] } = 0.$

\end{proof}

\section{Directed colored trees and plane binary trees}
One common way to show a set of elements spanning a module is a basis is to find another module and find a perfect pairing between them. Because our basis candidate is built from $\cR_n = \overline{\cG}_n,$ where $\overline{\cG}_n$ is the set of all two-colored trees in $\cG_n$ avoiding patterns $1r3r2, 2b1b3$ and $1r2b3,$ a natural object to use is the two-colored graphs. For convenience in defining the pairing, we add orientations onto the edges of the graphs.

\begin{defn}
An {\it oriented two-colored graph} is a two-colored graph whose edges have directions. We denote by $\OG_n$ the set of all oriented two-colored graphs on $X.$

We call an edge $i \to j$ {\it consistent} if $i < j$ and the color is red, or $i > j$ and the color is blue; and {\it inconsistent} otherwise.

An oriented two-colored graph is {\it consistent} if all of its edges are consistent.
\end{defn}

Although we can define pairing between a oriented two-colored tree in $\OG_n$ and a monomial in $\Lie_2(n)$ directly, it is easier to do so if we convert monomials in $\Lie_2(n)$ into combinatorial objects.

\begin{defn}
A {\it binary tree} is an ordered (rooted) tree, where all of its internal vertices, (i.e., vertices that are not leaves,) have exactly two children. See the Appendix of \cite{ec1} for a precise definition.

A {\it $2$v-colored binary tree} is a binary tree whose internal vertices are colored by red or blue. We denote by $\BT_n$ the set of all $2$v-colored binary trees whose leaves are labeled by $X.$

\end{defn}
One checks for any 2v-colored binary tree in $\BT_n,$ since the number of leaves is $|X|=n,$ it has exactly $n-1$ internal vertices.

\begin{rem}
Recall $M_n$ is the set of all the monomials in $\Lie_2(n).$
There is a canonical bijection between $\BT_n$ and $M_n:$ given a $2$v-colored binary tree, each leaf denotes a letter in $X,$ and we can construct a monomial in $M_n$ recursively by interpreting each internal vertex as a bracket of the left and right subtrees, with red vertices corresponding to $\brkone$ and blue vertices corresponding to $\brktwo.$

Because of this natural correspondence, we can consider $\cB_n(X)$ a subset of $\BT_n.$ In the rest of the paper, when we refer to $b_G$ as an element of $\BT_n,$ we mean the corresponding binary tree of $b_G \in \cB_n(X).$
\end{rem}

\begin{ex}[Example of the bijection between $\BT_n$ and $M_n$]
The left graph in Figure \ref{btqbt} shows the 2v-colored binary tree in bijection to the monomial $\langle [x_2, x_3], x_1 \rangle.$
\begin{figure}
\begin{center}
 
\input{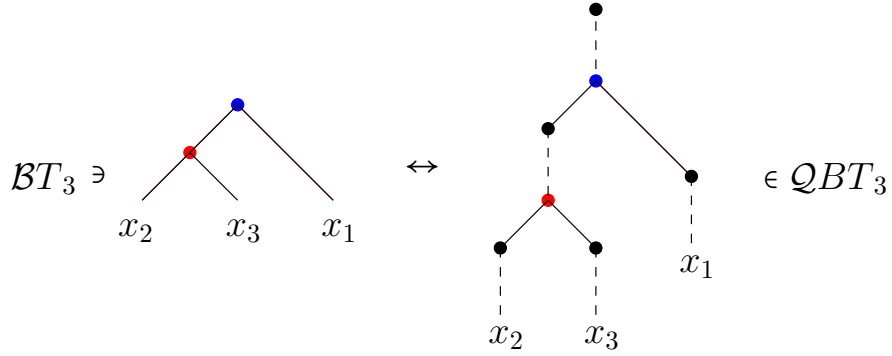}
 
\caption{The 2v-colored binary tree and quasi-binary tree correspond to the monomial $\langle [x_2, x_3], x_1 \rangle.$}
\label{btqbt}
\end{center}
\end{figure}

\end{ex}

The following definition of a pairing between $\BT_n$ and $\OG_n$ is an analogue of \cite{sinha} by Dev Sinha.

\begin{defn}\label{pairing}
Given a $2$v-colored binary tree $T$ in $\BT_n$ and an oriented two-colored graph $G$ in  $\OG_n,$ define 
$$\beta_{G,T}: \{ \mbox{edges of $G$} \} \to \{ \mbox{internal vertices of $T$} \}$$
by sending an edge $e: i \to j$ in $G$ to the nadir of the unique (simple) path $p_T(e)$ from $i$ to $j$ on $T,$ where the {\it nadir} of a path on a rooted tree is defined to be the (internal) vertex on the path that is closest to the root. Let $\tau_{G,T} = (-1)^N,$ where $N$ is the number of edges $e$ in $G$ for which $p_T(e)$ travels counterclockwise at its nadir. We say $\beta_{G,T}$ is {\it color-preserving} if for any edge $e \in G,$ the color of $e$ is the same as the color of $\beta_{G,T}(e).$

Define the {pairing} of $G, T$ as
$$\llangle  G, T \rrangle =
\begin{cases}\tau_{G,T}, & \mbox{if $\beta_{G,T}$ is a bijection and is color-preserving;} \\
0, & \mbox{Otherwise.}
\end{cases}$$
\end{defn}

\begin{defn}
Let $\Theta_n$ be the free $R$-module generated by the $2$v-colored binary trees in $\BT_n$ and $\Gamma_n$ be the free $R$-module generated by the oriented two-colored graphs in $\OG_n.$ Extend the pairing $\llangle \ , \ \rrangle$ of Definition \ref{pairing} to one between $\Theta_n$ and $\Gamma_n$ by linearity.
\end{defn}

$\Theta_n$ is not isomorphic to $\Lie_2(n),$ because we did not define relations between the elements of $\BT_n.$ We now define a submodule of $\Theta_n$ which corresponds to the relations in $\Lie_2(n).$

\begin{defn} \label{bcomb}
For brevity, given a binary tree $T \in \BT_n,$ we call the subtree of $T$ below the left child of the root of $T$ the {\it left subtree} of $T$ and denote it by $ls(T).$ Similarly, we define the {\it right subtree} of $T$ and denote it by $rs(T).$

\begin{alist}
\itm A {\it symmetry combination} in $\Theta_n$ is the sum of two binary trees $T_1, T_2 \in \BT_n$ where there exists a subtree $S_1$ of $T_1$ such that one obtains $T_2$ from $T_1$ by switching the left subtree and right subtree of $S_1.$ We say a symmetry combination is of type (S1) or (S2), 
depending on the color of the root of $S_1.$ 


\itm A {\it Jacobi combination} in $\Theta_n$ is the sum of three binary trees $T_1, T_2, T_3 \in \BT_n$ where there exists a subtree $S_1$ of $T_1$ such that the color of the root of $S_1$ is the same as the color of the right child of the root of $S_1,$ and one can obtain $T_2$ and $T_3$ by cyclic rotation of $ls(S_1),$ $ls(rs(S_1))$ and $rs(rs(S_1)).$ In other words, if we name the subtrees of $T_2$ and $T_3$ corresponding to $S_1$ of $T_1$ to be $S_2$ and $S_3,$ we have $ls(S_1) = ls(rs(S_3)) = rs(rs(S_2)),$ etc. We say a Jacobi combination is of type (J1) or (J2),
depending on the color of the root of $S_1$.


\itm A {\it mixed Jacobi combination} in $\Theta_n$ is corresponding to the mixed Jacobi identity (MJ) in $\Lie_2(n).$ It can be obtained by summing two copies of a Jacobi combination of type (J1) and changing the color of the right child of the root of $S_i$'s in the first copy and the color of the root of $S_i$'s in the second copy from red to blue, where $S_i$'s are the involving subtrees of $T_i$'s in the copy of Jacobi combination we use. We say this combination is of type (MJ).

\end{alist}

Figure \ref{btrel} demonstrates the combinations of types (S1), (J2), and (MJ) in $\Theta_n.$

Let $J_n \subset \Theta_n$ be the submodule generated by symmetry combinations, Jacobi combinations and mixed Jacobi combinations.
\end{defn} 

Note that $J_n$ is in fact same as the relation set $Rel_n$ we have used in last section. Since we use them in different contexts, we give them different names.

\begin{figure}
\begin{center}
 
\input{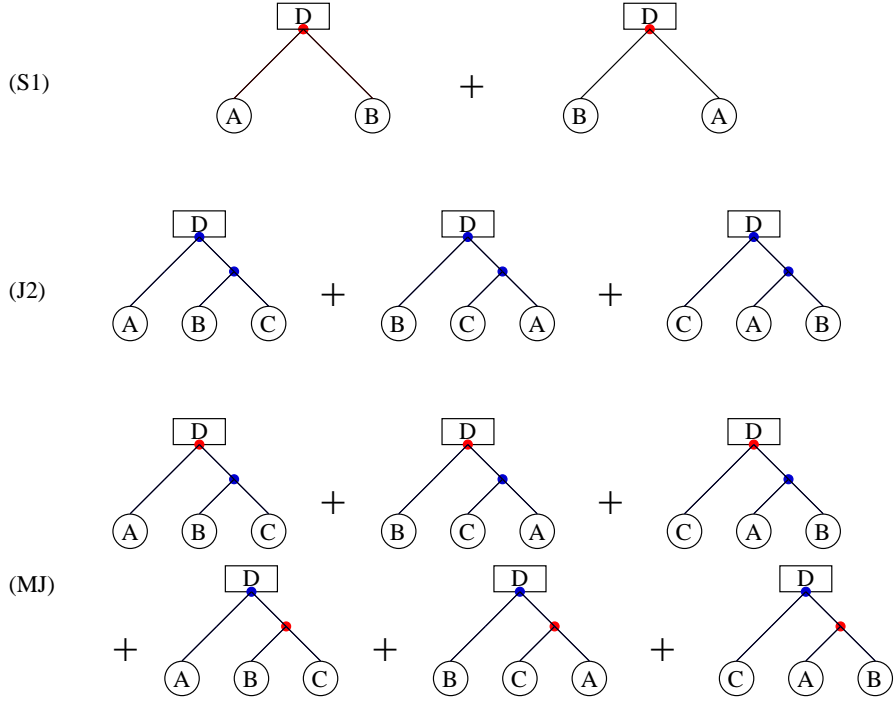}
 
\caption{Examples of elements that generate $J_n$}
\label{btrel}
\end{center}
\end{figure}

Because of the correspondence between the monomials in $\Lie_2(n)$ and the binary trees in $\BT_{n},$ and the correspondence between the relations in $\Lie_2(n)$ and the generators of $J_n,$ the following lemma immediately follows.

\begin{lem}
$$\Lie_2(n) \cong \Theta_n/ J_n.$$
\end{lem}

\begin{prop}\label{vanish1}
The pairing $\llangle \beta, \alpha \rrangle$ vanishes whenever $\alpha \in J_n.$
\end{prop}

\begin{proof}
It is sufficient to check the case that $\beta = G$ is a graph in $\OG_n$ and  $\alpha$ is one of the combinations defined in Definition \ref{bcomb}. When $\alpha = T_1 + T_2$ is a symmetry combination, it is clear that $\llangle G, T_1 \rrangle$ and $\llangle G, T_1 \rrangle$ either both are equal to zero or both are nonzero and only differ by a sign. Suppose $\alpha = T_1 + T_2 + T_3$ is a Jacobi combination. Without loss of generality, we assume $\alpha$ is of type (J2). Hence, we can consider $T_1, T_2, T_3$ to be the three trees in the sum labeled (J2) in Figure \ref{btrel}, and in the same order. If the pairing between $G$ and each of the three trees is zero, then there is nothing to check. Suppose at least one of them is nonzero. Let $X_A, X_B$ and $X_C$ be the labels of the subtrees $A, B$ and $C,$ respectively, and $X_D = X \setminus (X_A \cup X_B \cup X_C)$. If $X_D$ is empty, we let $G' = G;$ otherwise, there exists an edge $e$ in $G$ such that after removing $e,$ graph $G$ breaks into two graphs $G_D$ and $G'$ on vertices $X_D$ and $X_A \cup X_B \cup X_C.$ Now we must be able to find two edges $e_1$ and $e_2$ in $G',$ such that after removing these two edges, $G'$ breaks into three graphs $G_A, G_B, G_C$ on vertices $X_A, X_B, X_C,$ respectively. Two edges connecting three graphs implies that one of the graphs is connected to both edges. Without loss of generality, we assume $G_A$ is connected to both edges. Therefore, we can assume that $e_1$ connects $G_A$ and $G_B$ and $e_2$ connects $G_A$ and $G_C.$ One checks that the $\llangle G, T_1 \rrangle = 0$ since  $p_{T_1}(e_1)$ and $p_{T_1}(e_2)$ have the same nadir, and $\llangle G, T_2 \rrangle = - \llangle G, T_3 \rrangle.$ If $\alpha$ is a mixed Jacobi combination, we can similarly show that $\llangle G, \alpha \rrangle$ vanishes.
\end{proof}

We now define relations on oriented graphs.
\begin{defn}\label{ocomb}
\begin{alist}
\itm A {\it symmetry combination} in $\Gamma_n$ is the sum of two graphs $G_1, G_2 \in \OG_n$ such that one obtains $G_2$ by switching the orientation of one edge $e$ in $G_1.$ To be consistent, we say a symmetry combination is of type (S1) or (S2) depending on the color of $e.$
\itm A {\it Jacobi combination} in $\Gamma_n$ is the sum of three graphs $G_1, G_2, G_3 \in \OG_n,$
where $G_i$ has subgraph $S_i$ for each $i,$ such that  $G_1\setminus S_1 = G_2 \setminus S_2 = G_3 \setminus S_3,$  $S_1$ is a graph with two same-colored edges $i \to j$ and $j \to k,$ for some $i, j, k \in X,$ and one can obtain $S_2,$ $S_3$ by cyclicly rotating $i, j, k.$ Again, we say a Jacobi combination is of type (J1) or (J2) depending on the color of the edges in $S_i.$
\itm A {\it mixed Jacobi combination} in $\Gamma_n$  is obtained by summing two copies of a Jacobi combination in $\Gamma_n$ of type (J1) and changing the color of one edge of $S_i$ in the first copy and the color of the other edge of $S_i$ in the second copy from red to blue. We say this combination is of type (MJ).
\end{alist}
Figure \ref{ogrel} demonstrates the subgraphs $S_i$'s in the combinations of types (S1), (J2), and (MJ) in $\Gamma_n$. 

Let $I_n \subset \Gamma_n$ be the submodule generated by symmetry combinations, Jacobi combinations and mixed Jacobi combinations, as well as the graphs with more than one edge between two vertices and disconnected graphs.
\end{defn}

\begin{figure}
\begin{center}
\input{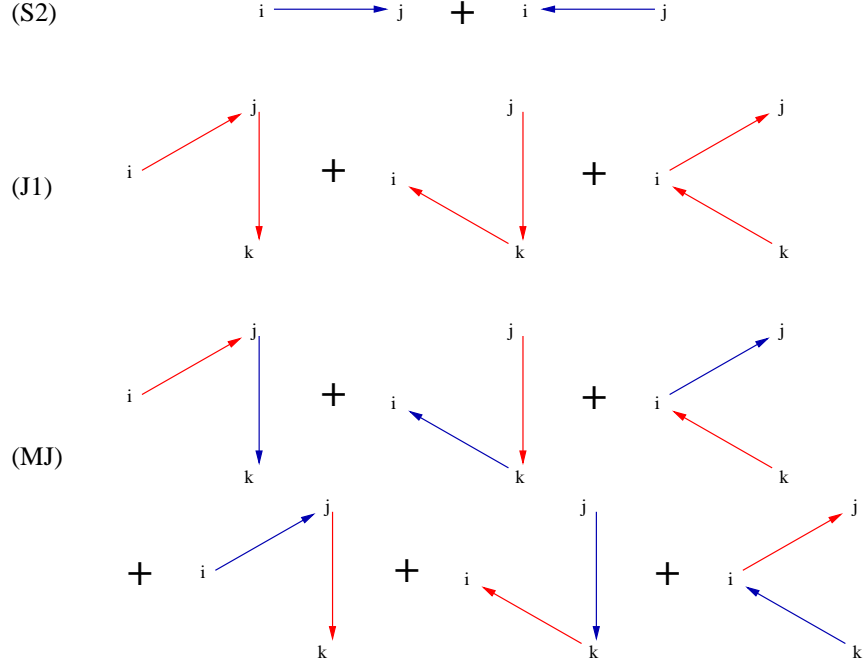}
 
\caption{Examples of (subgraphs of) elements that generate $I_n$}
\label{ogrel}
\end{center}
\end{figure}

\begin{prop}\label{vanish2}
The pairing $\llangle \beta, \alpha \rrangle$ vanishes whenever $\beta \in I_n.$
\end{prop}

\begin{proof} 
It is sufficient to check when $\alpha = T$ is a binary tree in $\BT_n$ and $\beta$ is an oriented graph $G$ with multiple edges between two vertices, a disconnected graph, or one of the combinations defined in Definition \ref{ocomb}. If $\beta = G$ with multiple edges between two vertices, or $G$ is disconnected, then $\beta_{G,T}$ cannot be a bijection. Thus, $\llangle G, T \rrangle = 0.$ When $\beta = G_1 + G_2$ is a symmetric combination, one checks that $\llangle G_1, T \rrangle$ and $\llangle G_2, T \rrangle$ either are both zero or only differ by a sign. If $\beta$ is a Jacobi combination, without loss of generality, we assume $\beta$ is of type (J1) in $\Gamma_n.$ Hence, we assume $\beta$ is the sum of three oriented two-colored graphs which  differ only on the subgraphs shown in (J1) in Figure \ref{ogrel}, and call them $G_1, G_2,$ and $G_3$ by order.  Let $v_{i,j}, v_{j,k}, v_{k,i}$ be the nadirs of the paths $p_T(i \to j), p_T(j \to k), p_T(k \to i),$ respectively. It is easy to see two of $v_{i,j}, v_{j,k}, v_{k,i}$ must agree. Without loss of generality, we assume $v_{i,j} = v_{j,k}$ Then $\llangle G_1, T \rrangle = 0$ and  $\llangle G_2, T \rrangle = -\llangle G_3, T \rrangle.$ If $\beta$ is a mixed Jacobi combination, we can similarly show that $\llangle \beta, T \rrangle$ vanishes.
\end{proof}

\begin{defn}
Let $$\Eil_2(n) = \Gamma_n/I_n.$$
\end{defn}

By Propositions \ref{vanish1} and \ref{vanish2},  the pairing  $\llangle , \rrangle$ between $\Theta_n$ and $\Gamma_n$ induces a pairing between $\Lie_2(n)$ and $\Eil_2(n).$ We still use the same notation  $\llangle , \rrangle$ to denote the pairing. We will show in one of the main results in Section \ref{perfpair} that the pairing $\llangle , \rrangle$ between $\Lie_2(n)$ and $\Eil_2(n)$ is perfect.

We have already shown that $\cB_n(X) = \{ b_G \ | \ G \in \overline{\cG}_n \}$ spans $\Lie_2(n).$ In the next section, we will define a set $\cO_n(X)$ and show it spans $\Eil_2(n).$ Then, in Section \ref{perfpair}, we show the matrix associated to the pairing between $\cB_n(X)$ and $\cO_n(X)$ is upper triangular to conclude our theorems.

\section{A basis candidate for $\Eil_2(n).$}

The elements of $\cB_n(X)$ are obtained from elements in $\overline{\cG}_n.$ It is natural to define a set in $\Gamma_n$ from $\overline{\cG}_n$ as well. Because of the relation between two-colored graphs and oriented two-colored graphs, we give the following definition.

\begin{defn} 
\begin{alist}
\itm For any oriented two-colored graph $G \in \OG_n,$ we define the {\it unoriented copy} of $G,$ denoted by $n_G,$ to be the two-colored graph obtained from $G$ by removing the orientations on the edges of $G.$

We say $G$ is a {\it tree} on $X,$ if $n_G$ is a tree on $X,$ i.e., $n_G$ is connected and acyclic. (Note $n_G$ is connected if and only if $G$ is connected. Therefore, the condition can be replaced by ``$G$ is connected and $n_G$ is acyclic''.)

\itm For any two-colored graph $G \in {\cG}_n,$ we define the {\it oriented copy} of $G,$ denoted by $o_G,$ to be the unique consistent oriented two-colored graph obtained from $G.$ In other words, for any edge $e = \{i,  j\}$ with $i < j$ in $G,$ we orient it as $i \to j$ if it is red,  and orient it as $j \to i$ if it is blue, and call the resulting oriented graph $o_G.$

In particular, we define $\cO_n(X)$ to be the set of all oriented copies of graphs in $\overline{\cG}_n = \cR_n:$ $$\cO_n(X) = \{ o_G \ | \ G \in \overline{\cG}_n \}.$$
\end{alist}
\end{defn}

We state without proof in the following lemma some easy results on the objects we just defined.
\begin{lem}\label{propogra}
\begin{ilist}
\itm For any $G_1, G_2 \in \OG_n,$ if the unoriented copy of $G_1$ is the same as the unoriented copy of $G_2,$ then $G_1$ is equal to $G_2$ or only differs by a sign in $\Eil_2(n).$
\itm The map $G \mapsto o_G$ gives a bijection between $\cG_n$ and the set of all consistent graphs in $\OG_n.$
\itm For any oriented two-colored graph $G \in \OG_n,$ $G$ is in $\cO_n(X)$ if and only if $G$ is a consistent tree on $X$, and there is a unique {\bf source} $r$ in $G,$ i.e., $r$ is the unique vertex in $G$ without incoming edges. (One can check that $r$ is in fact the root of $n_G.$)
\end{ilist}
\end{lem}

\begin{prop}\label{eilspan} $\cO_n(X)$ spans $\Eil_2(n).$
\end{prop}

We break the proof of this proposition into the following two lemmas.
\begin{lem}\label{cyclic}
For any oriented two-colored graph $G$, we have $G = 0$ in $\Eil_2(n)$ unless $G$ is a tree on $X,$ i.e., $G$ is connected and $n_G$ is acyclic. 
\end{lem}

\begin{lem}\label{lcgraph}
For any oriented two-colored graph $G$, if $G$ is a tree on $X,$ then $G,$ as an element in $\Eil_2(n),$ can be written as a linear combination of elements in $\cO_n(X).$
\end{lem}

It is clear that Proposition \ref{eilspan} follows from Lemma \ref{cyclic} and Lemma \ref{lcgraph}.

\begin{proof}[Proof of Lemma \ref{cyclic}]  
If $G$ is disconnected, then $G \in I_n,$ thus is $0$ in $\Eil_2(n).$ Hence, it is left to show that  if $n_G$ has a cycle $(i_1, i_2, \dots, i_k),$ then $G =0$ in $\Eil_2(n).$ 
We prove this by induction on the size $k$ of the cycle. If $k = 2,$ then there are at least two edges connecting some vertices $i_1$ and $i_2$ in $G.$ Thus, $G \in I_n$ and $G = 0$ in $\Eil_2(n).$ Suppose the lemma holds when $k < k_0$ (where $k_0 \ge 3$); we consider $k = k_0.$ Because of the symmetric combination,  we can assume the directions of the edges in the cycle in $G$ are $i_1 \to i_2 \to \cdots \to i_k \to i_1.$ If there are two edges consecutive in the cycle which have the same color, then without loss of generality we assume the edges $(i_1 \to i_2)$ and $(i_2 \to i_3)$ have the same color. Let $G'$ be the graph obtained from $G$ by replacing edge $(i_1  \to i_2)$ with $(i_3 \to i_1)$ and $G''$ be the graph obtained from $G$ by replacing edge $(i_2  \to i_3)$ with $(i_3 \to i_1),$ where for both cases we keep the color of the edges. Then $G + G' + G''$ is a Jacobi combination in $\Gamma_n.$ Thus, $G + G' + G'' =0$ in $\Eil_2(n).$ However, both of $n_{G'}$ and $n_{G''}$ have cycles of size $< k_0.$ Hence, $G = 0$ in $\Eil_2(n).$ If any two consecutive edges in the cycle have different colors, then by using the mixed Jacobi combination on the vertices $i_1, i_2$ and $i_3,$ we can show $G$ plus five graphs is $0$ in $\Eil_2(n),$ where the unoriented copy of each of these five graphs either has a cycle of size smaller than $k_0$, or has a cycle of size $k_0$ with consecutive same-colored edges. Therefore, by using the induction hypothesis together with the first case we proved, $G = 0$ in $\Eil_2(n).$
\end{proof}

\begin{proof}[Proof of Lemma \ref{lcgraph}] 
We prove the lemma by induction on $n.$ When $n=1,$ it is trivial. Now assuming the lemma is true when the size of the alphabet is smaller than $n,$ we will prove the lemma is true when $|X|=n$ in three cases. Recall a {\it leaf} of a tree is a vertex connected to only one edge.  
The cases are the following.
\begin{ilist}
\itm There exists an edge $e$ with two ends $x$ and $y$ in $G$ such that $x$ is a leaf and the color of $e$ is red if $y < x$ and is blue if $y > x.$
\itm There does not exist an edge $e$ satisfying the conditions in (i). There exists an red edge $e$ with two ends $x$ and $y$ in $G$, such that $x = x_n.$
\itm There does not exist an edge $e$ satisfying the conditions in (i). All the edges in $G$ that adjacent to $x_n$ are blue. Let $e$ be one of them with two ends $x = x_n$ and $y.$
\end{ilist}

In all cases, we are going to use the following idea and notation: suppose $e$ is an edge in $G$ with two ends $x$ and $y.$ By removing $e,$ we divide $G$ into two trees $G_x$ and $G_y$ on alphabets $X_1$ and $X_2$ respectively, where $x \in X_1$ and $y \in X_2.$ Let $n_i$ be the size of $X_i,$ for $i = 1,2.$
Since $n_1 + n_2 = |X| = n$ and both of $X_1$ and $X_2$ are nonempty, we have $n_i< n$, for $i = 1,2.$ Thus, by the induction hypothesis, we can write $G_x$ and $G_y$ as linear combinations of elements in $\cO_{n_1}(X_1)$ and $\cO_{n_2}(X_2),$ respectively:
$$G_x = \sum_{G_1 \in \cO_{n_1}(X_1)} \alpha_{G_1} G_1, \ \ \ G_y = \sum_{G_2 \in \cO_{n_2}(X_2)} \beta_{G_2} G_2.$$
For convenience, given two disjoint graphs $G'$ and $G''$ and an oriented two-colored edge $e',$ where one end of $e'$ is in $G'$ and the other end of $e'$ is in $G'',$ we denote by $(G', e', G'')$ the graph obtained by adding $e'$ to connect $G'$ and $G''.$ With this notation, we have $G = (G_x, e, G_y).$ Moreover, 
\begin{equation}\label{recgraph}
G =  \sum_{G_1 \in \cO_{n_1}(X_1), G_2 \in \cO_{n_2}(X_2)}  \alpha_{G_1} \beta_{G_2} \ (G_1, e, G_2).
\end{equation}

We will apply this formula to each of the three cases with respect to the $e$ given in each case. Since changing the orientation of an edge only changes the sign of the involved formula, without loss of generality, we assume $e$ is a consistent edge. 

For case (i), because $G_x = x,$ equation \eqref{recgraph} becomes
$$G =  \sum_{G_2 \in \overline{\cG}_{n_2}} \beta_{G_2} \ (x,  e, G_2).$$
It is sufficient to check that each $(x,  e, G_2)$ is in $\cO_n(X).$ However, $G_2$ is in $\cO_{n-1}(X_2),$ so according to Lemma \ref{propogra}/(iii), $G_2$ is a consistent tree on $X_2$, and there is a unique source $r.$ Because $e$ is consistent, it is clear $(x,  e, G_2)$ is a consistent tree on $X$. By  Lemma \ref{propogra}/(iii), it is left to check that $(x,  e, G_2)$ has a unique source. But the only new vertex $(x,  e, G_2)$ has is $x,$ which is connected to $y$ by the edge $e.$ Since $e$ is consistent and $e$ is red if $y < x$ and is blue if $y > x,$ we can determine the orientation of $e$ is $y \to x.$ Hence, $x$ cannot be a new source. Therefore, $r$ is the unique source in $(x,  e, G_2)$.

For cases (ii) and (iii), because we have already proved case (i), if each $(G_1, e, G_2)$ appearing in formula \eqref{recgraph} falls into case (i), then we are done. In other words, we only need to show that we can write any $(G_1, e, G_2)$ that is not covered by (i) as a linear combination of elements in $\cO_n(X).$ Again by Lemma \ref{propogra}/(iii), for $i=1, 2,$ we have that $o_{G_i}$ is a consistent tree on $X_i,$ and there is a unique source $r_i.$ If $G_1$ has a leaf $u \neq x,$ let $\tilde{e}$ be the edge adjacent to $u$ in $G_1$ and $w$ be the other end of $\tilde{e}.$ Since $G_1$ is a consistent tree, the orientation of $\tilde{e}$ is $w \to u$ and the color of $\tilde{e}$ is red if $w < u$ and is blue if $w > u.$ Since $u \neq x,$ $u$ is still a leaf in the new graph $(G_1, e, G_2)$. Hence, $(G_1, e, G_2)$ is in case (i). Similarly, if $G_2$ has a leaf $u' \neq y,$ we have that $(G_1, e, G_2)$ is in case (i) as well. Therefore, the only possibility that $(G_1, e, G_2)$ is not covered by (i) happens if, for $i = 1$ and $2,$ $o_{G_i}$ is a tree with only one leaf $x$ or $y$. In other words, $o_{G_i}$ is a directed path from $r_i$ to $x$ or $y.$ Now, we will  deal with this situation separately for  cases (ii) and (iii).

For case (ii), recall $x_n$ is the largest letter in $X,$ so we have $x > y.$ Since $e$ does not satisfy the conditions in (i), we have that $x = x_n$ is not a leaf. Hence, $n_1 = |X_1| > 1$, and there exists a unique vertex $z$ in $G_1$ that is connected to $x$ by an edge $e'.$ Since $x=x_n > z$ and $G_1$ is consistent, we conclude that $e'$ is red. Because $e$ is consistent and is colored red, the orientation of $e$ is $y \to x.$ It is easy to see that $(G_1, e, G_2)$ consists of two directed paths from $r_1$ to $x$ and from $r_2$ to $x$, and the last edge on each path are the red edges $e' = (z \to x)$ and $e = (y \to x).$  We apply (J1) (and (S1)) defined in Definition \ref{ocomb} to the subgraph of $(G_1, e, G_2)$ that consists of edges $e$ and $e',$ and we get that $(G_1, e, G_2)$ is equal to a sum of two graphs that are both covered by case (i). 

For case (iii), similarly to case (ii), we can show that $(G_1, e, G_2)$ consists of two directed paths from $r_1$ to $y$ and from $r_2$ to $y$, and the last edge on the former path is the blue edge $e = (x \to y)$ and the last edge on the latter path is an edge $e' = (z \to y)$, for some $z \in X_2.$ The color of $e'$ can be either red or blue.  If $e'$ is blue, since $e' = (z \to y) \in G_2$ is consistent, we have $z > y.$ Similarly to before,  by applying (J2) (and (S2)) defined in Definition \ref{ocomb} to the subgraph of $(G_1, e, G_2)$ that consists of edges $e$ and $e',$ we get that $(G_1, e, G_2)$ is equal to a sum of two graphs that are both covered by case (i). If $e'$ is red, the fact that $e'$ is consistent implies that $z < y.$ Thus $z < y < x = x_n.$ Applying (MJ) (and symmetry combinations) defined in Definition \ref{ocomb} to the subgraph of $(G_1, e, G_2)$ that consists of edges $e$ and $e',$ we get that $(G_1, e, G_2)$ is equal to a sum of five graphs that are covered either by case (i) or by case (ii).

\end{proof} 

\section{A perfect pairing}\label{perfpair}

We have shown that $\cB_n(X)$ spans $\Lie_2(n)$ (Proposition \ref{liespan}) and $\cO_n(X)$ spans $\Eil_2(n)$ (Proposition \ref{eilspan}). We will show in this section that the matrix of the pairing  $\llangle , \rrangle$ between $\cB_n(X)$ and $\cO_n(X)$ is nonsingular, and then conclude the following theorem.
\begin{thm}\label{main3}
The pairing $\llangle , \rrangle$ between $\Lie_2(n)$ and $\Eil_2(n)$ is perfect.
\end{thm}

We first review some terminology related to orderings on a set. (See Chapter 3 in \cite{ec1} for details.)
\begin{defn}
Given a set $S,$ we say a binary relation $\le$ on $S$ is a {\it partial order} if it satisfies the following three axioms:
\begin{ilist}
\itm Reflexivity: for all $a \in S,$ $a \le a.$
\itm Antisymmetry: if $a \le b$ and $b \le a$, then $a = b.$
\itm Transitivity: if $a \le b$ and $b \le c,$ then $a \le c.$
\end{ilist}
 Given a partial order $\le$ on a set $S,$ we use the notation $a < b$ to mean $a \le b$ and $a \neq b.$

 A partial order on $S$ is a {\it total order} if any two elements $a, b$ of $S$ are {\it comparable}, i.e., either $a \le b$ or $b \le a.$ 
 Given two binary relations $\sim_1$ and $\sim_2$ on $S,$ we say $\sim_1$ is a {\it refinement} of $\sim_2$ if for any $a, b \in S,$ $a \sim_2 b$ implies $a \sim_1 b.$ 
 If a total order $\le_1$ is a refinement of a partial order $\le_2,$ we call $\le_1$ a {\it linear extension} of $\le_2.$
 
\end{defn}

Note that both $\cB_n(X)$ and $\cO_n(X)$ are indexed by $\overline{\cG}_n = \cR_n,$ so it is natural to give the following definition. 
\begin{defn} Given a total order $\le$ on $\overline{\cG}_n = \cR_n,$ suppose under $\le$, the two-colored graphs in $\overline{\cG}_n$ (or the rooted trees in $\cR_n$) are ordered as $G_1 < G_2 < \cdots < G_{n^{n-1}}.$ We define {\it the matrix of the pairing $\llangle , \rrangle$ with respect to $\le$} between $\cB_n(X)$ and $\cO_n(X)$, denoted by $\cM_{n}^\le$, to be the $n^{n-1} \times n^{n-1}$ matrix where the $(i,j)$-entry is given by $\llangle o_{G_i}, b_{G_j} \rrangle,$ for $1 \le i, j \le n^{n-1}.$ 
\end{defn}

To show $\llangle , \rrangle$ is perfect between $\Lie_2(n)$ and $\Eil_2(n)$, it is enough to show $\cM_n^\le$ is nonsingular for some/all total order(s) $\le$ on $\overline{\cG}_n = \cR_n.$ Our goal is to find a certain total order such that it is relatively easier to show $\cM_n^\le$ is nonsingular. In fact, what we will do is to find a partial order on $\overline{\cG}_n = \cR_n$ such that any linear extension of this partial order has the desired properties. 
The plan for the rest of the section is as follows: We will define two kinds of binary relations $\le_{\ind}$ (Definition \ref{partind}) and $\le_{\op}$ (Definition \ref{partop}) on $\overline{\cG}_n = \cR_n$, where $\le_{\ind}$ is easily shown to be a partial order and $\le_{\op}$ has a close relationship to the construction of $\cB_n(X)$ (Lemma \ref{vanish3}). We then show that $\le_{\ind}$ is a refinement of $\le_{\op}$ (Lemma \ref{refine}) and use the fact that $\le_{\ind}$ is a partial order to deduce that $\le_{\op}$ is a partial order as well. Finally, we will be able to show $\cM_n^\le$ is an upper triangular matrix when $\le$ is a linear extension of $\le_{\op},$ (Proposition \ref{nonsing}), which leads to our main theorems.

\begin{defn}\label{l-level}Suppose $G$ is a two-colored rooted tree in $\overline{\cG}_n = \cR_n$ with root $r.$ 
\begin{alist}
\itm For any edge $e = \{x, y\}$ in $G$, where $x$ is closer to the root $r$ than $y,$ (or equivalently, $x$ is the parent of $y$,) we define $x$ to be the {\it tail} of $e$, and $y$ to be the {\it head} of $e.$ Note that because we can consider $G$ as a rooted tree, the map ($e \mapsto $ head of $e$) gives a bijection between the set of edges in $G$ and the set of non-root vertices in $G.$ We define the inverse map $e,$ that is, for any non-root vertex $x,$ we denote by $e(x)$ the unique edge in $G$ such that the head of $e(x)$ is $x.$ In other words, $e(x)$ is the first edge in the unique path from $x$ to the root $r$ in $G.$
\itm We define the {\it index} of $G$ to be 
$\iota(G) = (k_1, k_2, \dots, k_n),$
where $$k_i = \begin{cases}
+1, &  \mbox{if $x_i$ is the head of a red edge in $G$;} \\
-1, &  \mbox{if $x_i$ is the head of a blue edge in $G$;} \\
0, & \mbox{if $x_i = r,$ is the root of $G$.}
\end{cases}
$$
We use the reverse lexicographic order to order the indices of graphs, i.e. $\iota(G) <_{\rlex} \iota(H)$ if and only if the rightmost nonzero entry in $\iota(G) - \iota(H)$ is negative. 
\itm We call a vertex $x$ an {\it $\ell$-level vertex} of $G$ if the length of the unique path from $x$ to the root $r$ is $\ell.$ 
For any non-root vertex $x,$ by removing the edge $e(x),$ we divide $G$ into two graphs. We denote by $G(x)$ and $\tilde{G}(x)$ the subgraphs including $x$ and $r,$ respectively. By convention, we define $G(r)$ to be the original graph $G.$
Note when consider $G(x)$ as a rooted tree, $x$ is its root. We call $G(x)$ an {\it $\ell$-level subgraph} of $G$ if $x$ is an $\ell$-level vertex.
\end{alist}
\end{defn}

\begin{defn}\label{partind}
We define a binary relation $\le_{\ind}$ on $\overline{\cG}_n = \cR_n$ recursively.

For any two graphs $G_1$ and $G_2$ in $\overline{\cG}_n = \cR_n$ with root $r_1$ and $r_2,$ 
if $n=1,$  then $G_1 = G_2 = x_1,$ and we define $G_1 \le_{\ind} G_2.$

If $n > 1,$ suppose we have defined a binary relation $\le_{\ind}$ on $\overline{\cG}_m = R_m$, for any $m < n.$ There are three situations where we define $G_1 \le_{\ind} G_2$.
\begin{ilist}
\itm If $G_1$ and $G_2$ have the same numbers of blue edges and red edges and $\iota(G_1) <_{\rlex} \iota(G_2),$ we define $G_1 \le_{\ind} G_2$.
\itm If $\iota(G_1) = \iota(G_2)$ (note this implies that $r_1 = r_2,$ and $G_1$ and $G_2$ have the same numbers of blue edges, and red edges), and the number of $1$-level vertices (and/or subgraphs) of $G_1$ is less than that of $G_2,$ then we define $G_1 \le_{\ind} G_2.$
\itm If $\iota(G_1) = \iota(G_2),$ and $G_1, G_2$ have same number of $1$-level vertices (and/or subgraphs), then for $i=1$ and $2,$ let $G_{i,1}, \dots, G_{i,k}$ be $1$-level subgraphs of $G_i.$ If for any $1 \le j \le k,$ we have that $G_{1,j}$ and $G_{2,j}$ are on a same vertex set, and also $G_{1,j} \le_{\ind} G_{2,j},$ then we define $G_1 \le_{\ind} G_2.$
\end{ilist}
\end{defn}

\begin{lem}
$\le_{\ind}$ is a well-defined partial order on $\overline{\cG}_n = \cR_n$.
\end{lem}
\begin{proof}
We prove the lemma by induction on $n$, the size of the alphabet $X.$ When $n=1,$ there is only one element in $\overline{\cG}_n,$ so $\le_{\ind}$ is well-defined. Assuming that for any alphabet of size smaller than $n,$ we have $\le_{\ind}$ well-defined, we check the case $|X|= n.$ 

It is clear that $\le_{\ind}$ is reflexive. 

If $G_1 \le_{\ind} G_2$ and $G_2 \le_{\ind} G_1$, by the definition of $\le_{\ind}$ we must have $\iota(G_1) = \iota(G_2),$ and $G_1, G_2$ have same number of $1$-level vertices (or/and subgraphs). In addition, suppose  $G_{i,1}, \dots, G_{i,k}$ are $1$-level subgraphs of $G_i$ for $i =1,2$; then $G_{1,j} \le_{\ind} G_{2,j}$ and $G_{2,j} \le_{\ind} G_{1,j}$ for any $1 \le j \le k.$ By the induction hypothesis, $G_{1,j}= G_{2,j}$ for all $j.$ In particular, $G_{1,j}$ and $G_{2,j}$ has the same root, say, $r_j.$ Since $\iota(G_1) = \iota(G_2),$ the edges connecting $r_j$ and $r$ in $G_1$ and $G_2$ have the same color. Therefore, $G_1 = G_2.$ Thus, $\le_{\ind}$ is antisymmetric.

If $G_1 \le_{\ind} G_2$ and $G_2 \le_{\ind} G_3,$ we discuss three possibilities. If $\iota(G_1) <_{\rlex} \iota(G_2)$ or $\iota(G_2) <_{\rlex} \iota(G_3),$ we must have $\iota(G_1) <_{\rlex} \iota(G_3)$. Thus, $G_1 \le_{\ind} G_3.$ Now we can assume $\iota(G_1) = \iota(G_2) = \iota(G_3).$ If  the number of $1$-level vertices of $G_1$ is less than that of $G_2$ or the number of $1$-level vertices of $G_2$ is less than that of $G_3,$ then similarly we have $G_1 \le_{\ind} G_3.$ If the numbers of $1$-level vertices of $G_1, G_2$ and $G_3$ are the same, then for all $j,$ $G_{1,j}, G_{2,j}, G_{3,j}$ are on the same vertex set, and $G_{1,j} \le_{\ind} G_{2,j}$ and $G_{2,j} \le_{\ind} G_{3,j},$ where $G_{i,j}$'s are the $1$-level subgraphs of $G_i.$ By the induction hypothesis, $G_{1,j} \le_{\ind} G_{3,j}.$ Hence, $G_1 \le_{\ind} G_3.$ Thus, $\le_{\ind}$ is transitive. 

Therefore, $\le_{\ind}$ is a well-defined partial order on $\overline{\cG}_n = \cR_n$.
\end{proof}


Let $X_1 \cup X_2$ be a disjoint partition of $X.$ Suppose $G_i \in {\cG}_{|X_i|}$ and $x_i \in X_i$ for $i=1,2.$ Let $e$ be an edge $\{x_1, x_2\}$ with some color (red or blue). We denote by $(G_1, e,   G_2)$ the graph obtained by adding $e$ to connect $G_1$ and $G_2.$ We state without proof in the following lemma two results on $(G_1, e, G_2)$ when $G_i \in \overline{\cG}_{|X_i|} = \cR_{|X_i|}$ for each $i.$

\begin{lem}\label{connectRT}
Suppose $G_i \in \overline{\cG}_{|X_i|} = \cR_{|X_i|}$ and $x_i \in X_i$ for $i=1,2.$ Let $e$ be an edge $\{x_1, x_2\}$ with some color (red or blue). Then we have the following results.
\begin{ilist}
\itm If $x_i$ is the root of $G_i$ for each $i,$ then the two-colored graph $(G_1, e,   G_2)$ is in $\overline{\cG}_n = \cR_n.$ In particular, when we consider it as a rooted tree, its root is $\min(x_1, x_2)$ if $e$ is red and is $\max(x_1, x_2)$ if $e$ is blue.
\itm If $x_1$ is a non-root vertex of $G_1$ and $x_2$ is the root of $G_2,$ then the two-colored graph $(G_1, e,   G_2)$ is in $\overline{\cG}_n = \cR_n$ if and only if $x_1 < x_2$ and $e$ is red, or $x_1 > x_2$ and $e$ is blue.
\end{ilist}
\end{lem}

Lemma \ref{connectRT}/(i) states that if we connect two two-colored rooted trees (on two disjoint alphabets) by adding an edge connecting their roots, then the new two-colored graph is still a two-colored rooted tree. We will use this result in the construction of the second partial order $\le_{\op}$ on $\overline{\cG}_n$.

\begin{defn}
Suppose $G$ is a two-colored graph (or rooted tree) in $\overline{\cG}_n = \cR_n$ with root $r.$ 
\begin{alist}
\itm Let $y$ be a non-root vertex of $G$ and $e = e(y).$ Let $e' = \{r,y\}$ with the same color as $e.$ We define {\it the graph operated from $G$ with respect to $y$} to be the graph $$H := (\tilde{G}(y), e', G(y)).$$
Figure \ref{oper} shows how we construct $H$ from $G$ and $y$. By Lemma \ref{connectRT}/(i), we have that $H$ is in $\overline{\cG}_n.$

\begin{figure}
\begin{center}
\input{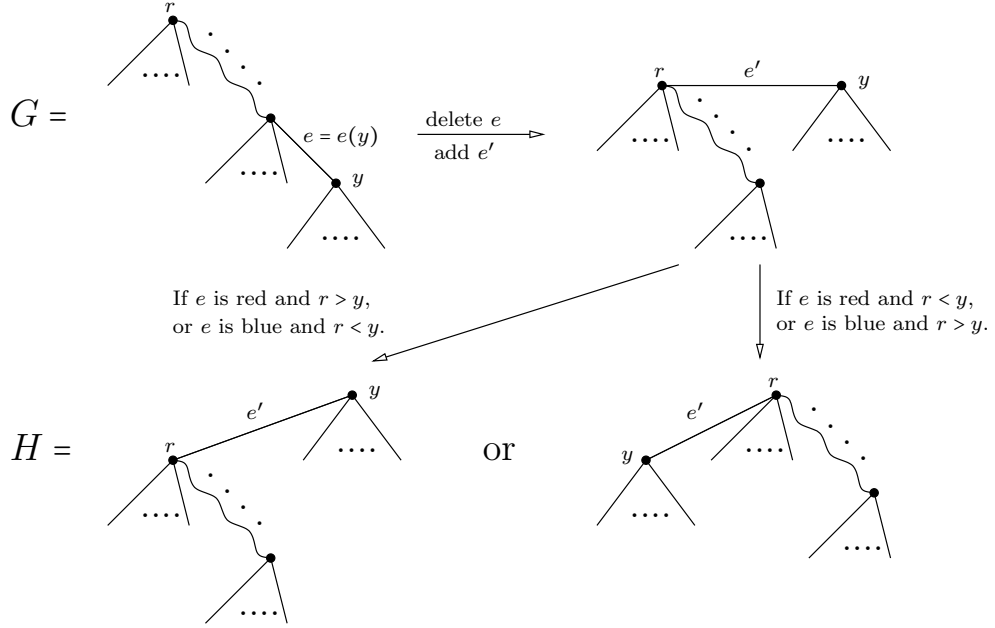}
\caption{$H$ is the graph operated from $G$ with respect to $y.$ Note that the color of $e'$ is the same as $e.$}
\label{oper}
\end{center}
\end{figure}



\itm We define a binary relation $\to_{\op}$ on $\overline{\cG}_n = \cR_n$ recursively. For any distinct two-colored graphs (or rooted trees) $G, H$ in $\overline{\cG}_n = \cR_n,$ we write $G \to_{\op} H$ if one of the following is satisfied.
\begin{itemize}
\item $H$ is operated from $G$ with respect to $y$, for some non-root vertex $y.$
\item If there exists a $1$-level vertex $x$ of $G,$ such that $H$ is obtained from $G$ by replacing the $1$-level subgraph $G(x)$ with $H',$ where $G(x) \to_{\op} H'.$ In other words, $H = (\tilde{G}(x), e', H'),$ where $e' $ is an edge connecting $r,$ the original root of $G,$ and the root $y$ of $H',$ with the same color as $e = e(x).$ Figure \ref{oper3} shows how we obtain $H$ from $G$ in this case. In Figure \ref{oper3}, we do not explicitly mark the root of $H$.  The root is determined by Lemma \ref{connectRT}/(i) as shown in Figure \ref{oper}.

\begin{figure}
\begin{center}
\input{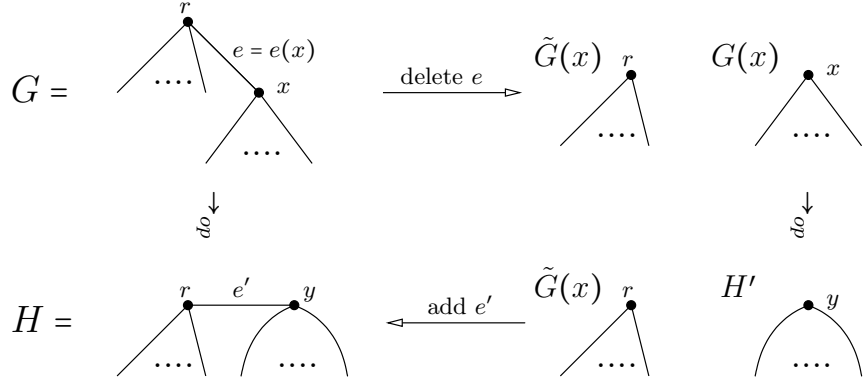}
\caption{$G \to_{\op} H,$  corresponding to a $1$-level vertex $x.$ Note the color of $e'$ is the same as $e.$}
\label{oper3}
\end{center}
\end{figure}
\end{itemize}
\end{alist}
\end{defn}

\begin{rem}\label{non1level}
In our definition of $G \to_{\op} H,$ we require $G$ and $H$ to be different graphs. Therefore, we do not have $G \to_{\op} G$ for any $G \in \overline{\cG}_n.$ Thus, if $H$ is operated from $G$ with respect to $y$ for some non-root vertex $y,$ then $G \to_{\op} H$ if and only if $y$ is not a $1$-level vertex. However, it is possible to modify our definition to include ``$G \to_{\op} G$'', i.e., removing the requirement that $G \neq H$. All the results related to this binary relation $\to_{\op}$ still hold under this modification. We choose to exclude ``$G \to_{\op} G$'' from our definition of $\to_{\op}$ to avoid trivial relations.
\end{rem}

\begin{ex} Figure \ref{exop} shows how the binary relation $\to_{\op}$ is defined on $\overline{\cG}_3 = \cR_3$ together with the index of each graph. 
\begin{figure}
\begin{center}
\input{exop.pstex_t}
\caption{The nine graphs in $\overline{\cG}_3 = \cR_3$, their indices, and the binary relation $\to_{\op}.$}
\label{exop}
\end{center}
\end{figure}
\end{ex}

\begin{defn}\label{partop}
We define a binary relation $\le_{\op}$ on $\overline{\cG}_n$ generated by $\to_{\op}:$ for any $G, H \in \overline{\cG}_n,$ $G \le_{\op} H$ if there exist $k \ge 0$ and a sequence of graphs $G_0 = G, G_1, \dots, G_{k-1}, G_k = H$ in $\overline{\cG}_n$ such that 
$$G_0 \to_{\op} G_1 \to_{\op} \cdots \to_{\op} G_{k-1} \to_{\op} G_k.$$
\end{defn}

\begin{rem}
Because we allow $k = 0$ in the definition of $\le_{\op},$ we have $G \le_{\op} G,$ for any $G \in \overline{\cG}_n.$ Therefore, if $H$ is operated from $G$ with respect to some non-root vertex $y,$ we always have $G \le_{\op} H,$ even if $y$ is a $1$-level vertex.
\end{rem}

\begin{lem}\label{wldfop}
$\le_{\op}$ is a well-defined partial order on $\overline{\cG}_n = \cR_n$. 
\end{lem}

Before we prove Lemma \ref{wldfop}, we will investigate the connection between the binary relation $\to_{\op}$ and the index function $\iota( \ )$ on two-colored rooted trees, and then conclude a relationship between $\le_{\op}$ and $\le_{\ind}.$

\begin{lem}\label{operprop}
If $G, H \in \overline{\cG}_n = \cR_n$ and $G \to_{\op} H,$ then we have the following:
\begin{ilist}
\itm $G$ and $H$ have the same number of red edges and blue edges.
\itm $\iota(G) \le_{\rlex} \iota(H).$ In particular, if $\iota(G) <_{\rlex} \iota(H),$ then $\iota(G)$ and $\iota(H)$ differ at exactly two coordinates.
\end{ilist}
\end{lem}
\begin{proof}
It is trivial to prove (i) by checking the definition of $\to_{\op}.$ We show (ii) by induction on $n$, the size of alphabet $X.$ Note that $\to_{\op}$ is only defined when $n \ge 3.$

When $n =3,$ one checks that the indices of the nine graphs in Figure \ref{exop} satisfy (ii). Now assume that $n > 3,$ and that (ii) holds when the size of $X$ is smaller than $n.$ 

Suppose $H$ is operated from $G$ with respect to $y$, for some vertex $y.$ If $e$ is red and $r >y$, or $e$ is blue and $r <y,$ then the root of $H$ is $y.$ The only difference between $\iota(G)$ and $\iota(H)$ is on the coordinates corresponding to $r$ and $y.$ One checks that $\iota(G) <_{\rlex} \iota(H)$ in both cases. If $e$ is red and $r < y$, or $e$ is blue and $r>y,$ then the root of $H$ is still $r,$ and it is easy to see that $\iota(G) = \iota(H).$ Thus, (ii) holds.

Suppose there exists a $1$-level vertex $x$ of $G,$ such that $H$ is obtained from $G$ by replacing the $1$-level subgraph $G(x)$ with $H',$ where $G(x) \to_{\op} H'.$ Let $y$ be the root of $H',$ and $r$ be the root of $G;$ then we have $H = (\tilde{G}(x), e', H'),$ where $e'=\{r, y\}$ has the same color as $e = e(x)$ in $G.$ The size of the alphabet of $G(x)$ and $H'$ is smaller than $n.$ By the induction hypothesis, we have $\iota(G(x)) \le_{\rlex} \iota(H'),$ and if $\iota(G(x)) <_{\rlex} \iota(H'),$ then $\iota(G(x))$ and $\iota(H')$ differ at exactly two coordinates. If $\iota(G(x)) = \iota(H'),$ then we have $x = y$ and $e' = e.$ Hence, the root of $H$ is still $r$ and $\iota(G) = \iota(H).$ Now we assume $\iota(G(x)) <_{\rlex} \iota(H'),$ and $\iota(G(x))$ and $\iota(H')$ differ at exactly two coordinates. If $x = y,$ then $e' = e$ and $r$ is the root of $H.$ Thus, $\iota(G)$ and $\iota(H)$ differ at exactly the same places as $\iota(G(x))$ and $\iota(H')$ differ. Therefore, (ii) holds.  If $x \neq y,$ it is clear that $\iota(G(x))$ and $\iota(H')$ differ at exactly the coordinates corresponding to $x$ and $y.$ The situation in this case is complicated; we only check the case when $x < y,$ and the case when $x >y$ is analogous. 

Since $x <y$, and $\iota(G(x))$ and $\iota(H')$ only differ at the coordinates corresponding to $x$ and $y$, we only write the two coordinates corresponding to $x$ and $y$ in the order of $(x, y)$ to present $\iota(G(x))$ and $\iota(H').$ Because $x$ and $y$ are the roots of $G(x)$ and $H'$ respectively, and $G(x)$ and $H$ have the same number of blue edges and red edges,  we have $\iota(G(x)) = (0, \epsilon),$ and $\iota(H') = (\epsilon, 0),$ where $\epsilon = \pm 1.$  We know $\iota(G(x)) = (0, \epsilon) <_{\rlex}(\epsilon, 0) = \iota(H').$ Hence $\epsilon = -1.$ Clearly, the only possible places $\iota(G)$ and $\iota(H)$ could differ are the coordinate corresponding to $r,$ $x$ and $y,$ so we only look at these three coordinates. We discuss the three cases according to the position of $r$ comparing with $x$ and $y.$ 
\begin{enumerate}
\item If $r < x < y,$ then the colors of $e$ and $e'$ are red, so $r$ is the root of $H.$ The coordinates of $\iota(G)$ and $\iota(H)$ corresponding to $r, x, y$ (in this order) are $(0, 1, -1)$ and $(0, -1, 1).$
\item If $x < r < y,$ then the colors of $e$ and $e'$ are blue, so $y$ is the root of $H.$ The coordinates of $\iota(G)$ and $\iota(H)$ corresponding to $x, r, y$ (in this order) are $(-1, 0, -1)$ and $(-1, -1, 0).$
\item If $x < y < r,$ then the colors of $e$ and $e'$ are blue, so $r$ is the root of $H.$ The coordinates of $\iota(G)$ and $\iota(H)$ corresponding to $x, y, r$ (in this order) are $(-1, -1, 0)$ and $(-1, -1, 0).$
\end{enumerate}
Therefore, in all cases, (ii) holds.

\end{proof}

\begin{lem}\label{refine}
$\le_{\ind}$ is a refinement of $\le_{\op}.$
\end{lem}
\begin{proof}

Because $\le_{\op}$ is generated by the relation $\to_{\op},$ and $\le_{\ind}$ is a partial order and thus is transitive, it is sufficient to show that for any $G, H \in \overline{\cG}_n = \cR_n,$ $$G \to_{\op} H \mbox{ implies } G \le_{\ind} H.$$ We will show this by induction on $n,$ the size of the alphabet $X.$ 
When $n = 1,2,$ $\to_{\op}$ is not defined between any two distinct graphs. Therefore, our assertion that $G \to_{\op} H \mbox{ implies } G \le_{\ind} H$ is tautologically true.
Assuming $G \to_{\op} H \Rightarrow G \le_{\ind} H$ when $|X| < n,$ we consider when $|X|=n.$ 
Given $G \to_{\op} H,$ by Lemma \ref{operprop}, we have $\iota(G) \le_{\rlex} \iota(H).$ If $\iota(G) <_{\rlex} \iota(H),$ together with Lemma \ref{operprop}/(i), we already have $G \le_{\ind} H.$ Therefore, we can assume $\iota(G) = \iota(H).$ Suppose $r$ is the root of $G.$

Suppose $H$ is the graph operated from $G$ with respect to some non-root vertex $y.$ Let $e := e(y)$ be the edge connecting $y$ and its parent in $G.$ It is clear that $\iota(G) = \iota(H)$ only when $e$ is red and $r < y$, or $e$ is blue and $r>y,$ where the root of $H$ is still $r.$ As we mentioned in Remark \ref{non1level}, if $y$ is a $1$-level vertex of $G,$ then $G = H$ and we do not have that $G \to_{\op} H.$ Hence, we assume that $y$ is not a $1$-level vertex of $G.$ Then the number of $1$-level vertices in $G$ is one less than that in $H,$ so $G \le_{\ind} H.$  Therefore, we have $G \le_{\ind} H.$

Suppose there exists a $1$-level vertex $x$ of $G$ such that $H$ is obtained from $G$ by replacing the $1$-level subgraph $G(x)$ with $H',$ where $G(x) \to_{\op} H'.$ The size of the alphabet of $G(x)$ and $H'$ is smaller than $n,$ so by the induction hypothesis, we have that $G(x) \le_{\ind} H'.$ Hence, we have $\iota(G) = \iota(H)$, all but one $1$-level subgraphs of $G$ and $H$ are the same, and the different ones are $G(x) \le_{\ind} H'.$ We conclude that $G \le_{\ind} H.$

\end{proof}

\begin{proof}[Proof of Lemma \ref{wldfop}] 
It is clear that $\le_{\op}$ is reflexive and transitive, so it's left to show it is antisymmetric. If $G \le_{\op} H$ and $H \le_{\op} G$, by Lemma \ref{refine}, we have $G \le_{\ind} H$ and $H \le_{\ind} G.$ Since $\le_{\ind}$ is a well-defined partial order, so is antisymmetric, we have $G = H.$ Therefore, $\le_{\op}$ is antisymmetric as well.
\end{proof}

\begin{lem}\label{vanish3}
For any $G, H \in \overline{\cG}_n = \cR_n,$
we have the following:
\begin{ilist} 
\itm $\llangle o_G, b_H \rrangle = 0$ unless $G \le_{\op} H.$
\itm $\llangle o_G, b_H \rrangle = \pm 1$ if $G = H.$
\end{ilist}
\end{lem}

We will show a stronger result (Proposition \ref{vanish4}) than this lemma and prove this lemma as a corollary of that result in the next section.




Assuming Lemma \ref{vanish3}, we have the following proposition.

\begin{prop}\label{nonsing} Let $\le$ be a linear extension of $\le_{\op}$ on $\overline{\cG}_n = \cR_n.$ Then $\cM_n^\le$ is an upper triangular matrix with invertible entries on the diagonal, and thus is nonsingular.
\end{prop}
\begin{proof}
By Lemma \ref{vanish3}/(i) and the definition of linear extension, we have that $\llangle o_G, b_H \rrangle = 0$ unless $G \le H.$ Hence, $\cM_n^\le$ is an upper triangular matrix. Lemma \ref{vanish3}/(ii) implies that the diagonal entries of $\cM_n^\le$ are $\pm 1,$ thus are invertible.
\end{proof}

Theorems \ref{main1} and \ref{main3} immediately follow from Proposition \ref{nonsing}. We also have the following corollaries.

\begin{cor}
$\cB_n(X)$ is a basis for $\Lie_2(n)$.
\end{cor}

\begin{cor}\label{eilrank}
$\Eil_2(n)$ is free of rank $n^{n-1},$ and $\cO_n(X)$ is a  basis for $\Eil_2(n).$
\end{cor}

\section{More bases for $\Lie_2(n)$}
In this section, we will show that we can obtain more bases for $\Lie_2(n)$ from $\overline{\cG}_n = \cR_n.$ We will discuss a property of $\cB_n(X)$ and show that having this particular property is enough to guarantee that a subset of $M_n$ (the set of all monomials in $\Lie_2(n)$) is a basis for $\Lie_2(n).$

The map $G \to b_G$ (defined in Definition \ref{defBn}) gives a map from $\overline{\cG}_n = \cR_n$ to $M_n.$ As we discussed earlier, we can consider $M_n$ and $\BT_n$ to be the same sets. We define a natural inverse map from $M_n = \BT_n$ to $\overline{\cG}_n = \cR_n$ as follows. Recall we defined the graphical root $\gr(m)$ of a monomial $m$ in Definition \ref{root}.

\begin{defn}\label{mapG1}
For any monomial $m \in M_n,$ we define {\it the two-colored graph corresponding to $m$}, denoted by $\sG(m)$, recursively:
\begin{ilist}
\itm If $m = x$ a single variable, let $\sG(m) := x;$
\itm If $m = \{m_1, m_2\},$ let $\sG(m) := (\sG(m_1), e, \sG(m_2)),$ where $e$ is an edge connecting $\gr(m_1)$ and $\gr(m_2)$ with color red if $\{ \cdot , \cdot \} = \brkone$ or blue if $\{ \cdot , \cdot \} = \brktwo.$
\end{ilist}
\end{defn}

\begin{ex}
Let $G$ be the second two-colored graph shown in the first row of Figure \ref{exbn}. When $m = [[x_1, x_3], x_2], [[x_3, x_1], x_2],$ or $[[x_1, x_2], x_3],$ we have $\sG(m) = G.$ 
\end{ex}

We have the following lemma about $\sG(m),$ which can be shown by induction on $n$ and by using Lemma \ref{connectRT}/(i) recursively. We omit the details of the proof.
\begin{lem}\label{sG}
For any monomial $m \in M_n = \BT_n,$ we have that $\sG(m)$ is in $\overline{\cG}_n = \cR_n.$ In particular, the root of $\sG(m)$ is exactly the graphical root $\gr(m)$ of $m.$
\end{lem}

Hence, the map $\sG: m \to \sG(m)$ gives a map from $M_n = \BT_n$ to $\overline{\cG}_n = \cR_n.$ If we restrict the map $\sG$ to the set $\cB_n(X) \subset M_n,$ it is clear that $\sG$ and $G \to b_G$ are inverse to one another. Hence, we have the following lemma.

\begin{lem}\label{basesprop}
The map $\sG$ induces a bijection between $\cB_n(X)$ and $\overline{\cG}_n.$
\end{lem}

\begin{ex}
In Figure \ref{exbn}, $\sG$ maps each monomial to the two-colored rooted tree shown above it. This demonstrates the bijection between $\cB_3(X)$ and $\overline{\cG}_3$ given by $\sG$.
\end{ex}

It turns out the property described in Lemma \ref{basesprop} is a sufficient condition for a subset of $M_n = \BT_n$ to be a basis for $\Lie_2(n).$

\begin{thm}\label{main4}
For any subset $S$ of $M_n = \BT_n$, if the map $\sG$ induces a bijection between $S$ and $\overline{\cG}_n,$ then $S$ is a basis for $\Lie_2(n).$

\end{thm}

\begin{rem}
For each $G \in \overline{\cG}_n,$ we let $\sG^{-1}(G)$ be the set of monomials $m \in M_n$ satisfying $\sG(m) = G.$ The condition that the map $\sG$ induces a bijection between $S$ and $\overline{\cG}_n$ is equivalent to having that $|S \cap \sG^{-1}(G)| = 1$ for each $G \in \overline{\cG}_n.$ 
\end{rem}

\begin{ex}$\cB_3(X),$ the set of the 9 monomials shown in Figure \ref{exbn}, is a basis for $\Lie_2(3).$ 

Let $G$ be the second two-colored rooted tree shown in the first row of Figure \ref{exbn}. Both $b_G = [[x_1, x_3], x_2]$ and $m = [[x_1, x_2], x_3]$ are in $\sG^{-1}(G).$ Therefore, if we let $S$ be the set obtained from $\cB_3(X)$ by replacing $b_G$ by $m,$  then by Theorem \ref{main4}, $S$ is a basis for $\Lie_2(3)$ as well.
\end{ex}

It will be shown later that Theorem \ref{main4} and Lemma \ref{vanish3} are corollaries of the following key result of this section.

\begin{prop}\label{vanish4}
For any $s \in M_n = \BT_n$ and any $G \in \overline{\cG}_n = \cR_n,$ we have the following:
\begin{ilist}
\itm $\llangle o_G, s \rrangle = 0$ unless $G \le_{\op} \sG(s).$
\itm $\llangle o_G, s \rrangle = \pm 1$ if $G = \sG(s).$
\end{ilist}
\end{prop}

Because the pairing $\llangle \ , \ \rrangle$ is defined between $\OG_n$ and $\BT_n,$ when we show Proposition \ref{vanish4}, it is more convenient if we consider $s$ to be an element in $\BT_n.$ Therefore, for easy reference, we rewrite Definition \ref{mapG1} in terms of $\BT_n$. 

\begin{defn}\label{mapG2}
For any 2v-colored binary tree $T \in \BT_n$ we define {\it the two-colored graph corresponding to $T$}, denoted by $\sG(T)$, recursively:
\begin{ilist}
\itm If $T = x$ has only one vertex, let $\sG(T) := x;$
\itm If $T_1$ and $T_2$ are the left subtree and right subtree of $T,$  let $\sG(T) := (\sG(T_1), e, \sG(T_2)),$ where $e$ is an edge connecting 
the roots of $\sG(T_1)$ and $\sG(T_2)$
with the same color as the root of $T.$ 
\end{ilist}
\end{defn}

\begin{rem}
It is easy to verify that Definition \ref{mapG2} is equivalent to Definition \ref{mapG1} when we consider $M_n = \BT_n.$ Therefore, we still have $\sG(T) \in \overline{\cG}_n = \cR_n.$ Also, by Lemma \ref{sG}, we are able to use ``roots of $\sG(T_1)$ and $\sG(T_2)$'' instead of ``$\gr(T_1)$ and $\gr(T_2)$'' in the description of the definition.

\end{rem}

We need the following lemma and its corollary to prove Proposition \ref{vanish4}.

\begin{lem}\label{opind1}
Let $X_1 \cup X_2$ be a disjoint partition of $X.$ Suppose $G_i$ and $H_i$ are in $\overline{\cG}_{|X_i|} = \cR_{|X_i|}$ with roots $x_i$ and $y_i$ in $X_i,$ for $i=1,2.$ Let $e = \{x_1, x_2\}$ be an edge connecting the roots of $G_1$ and $G_2$ of color $\kappa,$ where $\kappa$ is blue or red. Let $G := (G_1, e, G_2).$ Then we have the following:
\begin{ilist}
\itm If $G_1 \le_{\op} H_1,$ then $G \le_{\op} (H_1, e', G_2),$ where $e' = \{y_1, x_2\}$ is an edge connecting the roots of $H_1$ and $G_2$ of color $\kappa.$
\itm If $G_2 \le_{\op} H_2,$ then $G \le_{\op} (G_1, e', H_2),$ where $e' = \{x_1, y_2\}$ is an edge connecting the roots of $G_1$ and $H_2$ of color $\kappa.$
\itm If $G_1 \le_{\op} H_1$ and $G_2 \le_{\op} H_2,$ then $G \le_{\op} (H_1, e', H_2),$ where $e' = \{y_1, y_2\}$ is an edge connecting the roots of $H_1$ and $H_2$ of color $\kappa.$
\end{ilist}
\end{lem}

\begin{proof}
(iii) follows from (i) and (ii). Also, (i) and (ii) are symmetric. Hence, it is enough to show (i). Because $\le_{\op}$ is generated by $\to_{\op}$ and is transitive, it is sufficient to show (i) when we assume $G_1 \to_{\op} H_1.$ Let $H :=  (H_1, e', G_2).$ We discuss the two possibilities for the root of $G.$ 
\begin{itemize}
\item If the root of $G$ is $x_2,$ the root of $G_2,$ then the root of $G_1$ is a $1$-level vertex in $G.$ Thus, $G_1 \to_{\op} H_1$ implies that $G \to_{\op} H.$ So $G \le_{\op} H.$ 
\item If the root of $G$ is $x_1,$ the root of $G_1,$ then by Lemma \ref{connectRT}/(i), we have that $x_1 < x_2$ and $e$ is red, or $x_1 > x_2$ and $e$ is blue. 
Let $H' = (H_1, e, G_2).$  Note that the root of $H_1$ is not necessarily to be $x_1.$ If the root of $H_1$ is $x_1,$ then $H' \in \overline{\cG}_n$ by Lemma \ref{connectRT}/(i); otherwise, we still have $H' \in \overline{\cG}_n$ according to Lemma \ref{connectRT}/(ii). One checks that $G_1 \to_{\op} H_1$ implies that $G \to_{\op} H'.$ However, $H$ is operated from $H'$ with respect to $x_2.$ Therefore, $G \le_{\op} H.$
\end{itemize}

\end{proof}

\begin{cor}\label{opind2}
For any $G \in \overline{\cG}_n = \cR_n,$ let $y$ be a non-root vertex of $G$ and $e=e(y).$ Suppose the color of $e$ is $\kappa,$ where $\kappa$ is blue or red. Let $G_1 = \tilde{G}(y)$ and $G_2=G(y)$  be the graphs obtained from $G$ by removing the edge $e.$ Let $H_i$ be a two-colored rooted tree on the same alphabet as $G_i,$ for $i=1,2.$ Let $H = (H_1, e', H_2),$ where $e'$ is an edge connecting the roots of $H_1$ and $H_2$ of color $\kappa.$

If $G_1 \le_{\op} H_1$ and $G_2 \le_{\op} H_2,$ then $G \le_{\op} H.$
\end{cor}

\begin{proof}
Let $H' = (G_1, e'', G_2),$ where $e''$ is an edge connecting the roots of $G_1$ and $G_2$ of color $\kappa.$ $H'$ is operated from $G$ with respect to $y.$ Thus, $G \le_{\op} H'.$ But by Lemma \ref{opind1}, we have that $H' \le_{\op} H.$ Thus, $G \le_{\op} H.$
\end{proof}

\begin{proof}[Proof of Proposition \ref{vanish4}]

If $n=1,$ it is trivial to check that the proposition is true. Hence, we assume $n \ge 2.$

We prove (i) first. $\llangle o_G, T \rrangle = 0$ unless $\beta_{o_G, T}$ is a color-preserving bijection. It is enough to show that
$$\beta_{o_G, T}: \{ \mbox{edges of $o_G$} \} \to \{ \mbox{internal vertices of $T$} \}$$
being a color-preserving bijection implies that $G \le_{\op} \sG(T).$ Recall the map $\beta_{o_G,T}$ is defined in Definition \ref{pairing}. In fact, the definition of this map has nothing to do with the orientation of the edges of $o_G.$ Therefore, we can define an equivalent map in terms of $G,$ the unoriented copy of $o_G:$
$$\tilde{\beta}_{G,T}: \{ \mbox{edges of $G$} \} \to \{ \mbox{internal vertices of $T$} \}$$
sends an edge $e = \{i, j\}$ in $G$ to the nadir of the shortest path $p_T(e)$ between $i$ and $j$ on $T.$ 
Our goal becomes to show that 
\begin{equation*}
\tag{$\triangle$} \tilde{\beta}_{G,T} \mbox{ is a color-preserving bijection } \Rightarrow G \le_{\op} \sG(T).
\end{equation*}
We will show $(\triangle)$ by induction on $n,$ the size of the alphabet $X.$ When $n= 2,$ one checks $ \tilde{\beta}_{G,T}$ is a color-preserving bijection if and only if  $G = \sG(T).$ Assuming $(\triangle)$ holds when $|X| < n,$ we will show $(\triangle)$ holds when $|X|=n.$ Suppose $\tilde{\beta}_{G,T}$ is a color-preserving bijection. Let $e = \{x, y\}$ be the edge of $G$ that is in color-preserving bijection with the root of $T$ under $\tilde{\beta}_{G,T}.$ Without loss of generality, we assume $x$ is the parent of $y$ in $G.$ Let $G_1 = \tilde{G}(y)$ and $G_2 = G(y)$ be the two graphs obtained by removing the edge $e$ in $G.$ Suppose $G_i$ is on alphabet $X_i,$ for $i=1,2.$ Let $T_1$ and $T_2$ be the left subtree and the right subtree of $T$. Without loss of generality,  we assume $x$ is a leaf of $T_1$ and $y$ is a leaf of $T_2,$ respectively. Since $\tilde{\beta}_{G,T}$ is a color-preserving bijection, we must have, for $i=1$ and $2,$ that the leaves of $T_i$ are labeled by $X_i$, and $\tilde{\beta}_{G_i,T_i}$ is a color-preserving bijection. The size of $X_i$ is smaller than $n,$ so by the induction hypothesis, $G_i \le_{\op} \sG(T_i),$ for $i =1,2.$ Applying Corollary \ref{opind2}, we get $G \le_{\op} \sG(T).$

Now we will prove (ii). Note that $\llangle o_G, T \rrangle = \pm 1$ if and only if $\beta_{o_G, T}$ is a color-preserving bijection if and only if $\tilde{\beta}_{G, T}$ is a color-preserving bijection. Hence, it is enough to show that 
\begin{equation*}
\tag{$\square$} G = \sG(T)  \Rightarrow  \tilde{\beta}_{G,T} \mbox{ is a color-preserving bijection}.
\end{equation*}
We show $(\square)$ by induction on $n.$ As we stated earlier, when $n= 2,$ we have that $ \tilde{\beta}_{G,T}$ is a color-preserving bijection if and only if  $G = \sG(T).$ Assuming $(\square)$ holds when $|X| < n,$ we will show $(\square)$ holds when $|X|=n.$ We still let $T_1$ and $T_2$ be the left subtree and the right subtree of $T.$ Let $G_1 = \sG(T_1)$ with root $r_1,$ $G_2 = \sG(T_2)$ with root $r_2,$ and $e = \{r_1, r_2\}$ is an edge with the same color as the root of $T.$ Then $G = (G_1, e, G_2).$ By induction hypothesis, $\tilde{\beta}_{G_i,T_i}$ is a color-preserving bijection, for $i =1,2.$ Let $e'$ be an edge in $G.$ If $e'$ is in $G_i$ for $i=1$ or $2,$ then $\tilde{\beta}_{G,T}$ sends $e'$ to $\tilde{\beta}_{G_i,T_i}(e');$ otherwise, $e' = e = \{r_1, r_2\},$ then $\tilde{\beta}_{G,T}$ sends $e'$ to the root of $T,$ which has the same color as $e'.$ Therefore, $\tilde{\beta}_{G,T}$ is a color-preserving bijection.

\end{proof}

Lemma \ref{vanish3} follows from Lemma \ref{basesprop} and Proposition \ref{vanish4}. Because we had assumed Lemma \ref{vanish3} in the proof of Proposition \ref{nonsing}, only now can we consider the proofs of Theorem \ref{main1} and Theorem \ref{main3}, as well as of the two corollaries stated at the end of the last section, to be truly complete.

Since we know that the rank of $\Lie_2(n)$ is $n^{n-1}$, for any $(n^{n-1})$-subset of monomials $S$ of $M_n,$ if the matrix of the pairing $\llangle , \rrangle$ between $S$ and $\cO_n(X)$ is nonsingular, then we can conclude that $S$ is a basis for $\Lie_2(n).$ Using this observation, we are able to prove Theorem \ref{main4}.

\begin{proof}[Proof of Theorem \ref{main4}]
$\sG$ induces a bijection between $S$ and $\overline{\cG}_n,$ so the cardinality of $S$ is $n^{n-1}.$ Also, we can index the elements in $S$ by $\overline{\cG}_n:$ $$S = \{ s_G \ | \ G \in \overline{\cG}_n \},$$ where $s_G$ is the element in $S$ that maps to $G$ under $\sG,$ i.e., $\sG(s_G) = G.$

Let $\le$ be a linear extension of $\le_{\op}$ on $\overline{\cG}_n.$ Suppose under $\le,$ the graphs in $\overline{\cG}_n$ are ordered by $G_1 < G_2 < \cdots < G_{n^{n-1}}.$ We define the matrix of the pairing $\llangle , \rrangle$ with respect to $\le$ between $S$ and $\cO_n(X)$, denoted by $\cM_{n}^\le(S)$, to be the $n^{n-1} \times n^{n-1}$ matrix where the $(i,j)$-entry is given by $\llangle o_{G_i}, s_{G_j} \rrangle,$ for $1 \le i, j \le n^{n-1}.$ Similarly to the proof of Proposition \ref{nonsing}, we can show that $\cM_{n}^\le(S)$ is an upper triangular matrix with invertible entries on the diagonal. Hence, $\cM_{n}^\le(S)$ is nonsingular. Therefore, $S$ is a basis for $\Lie_2(n).$
\end{proof}

\section{Equivalence of Theorem \ref{main1} and Theorem \ref{main2}}
In this section, we will establish a connection between the bases for $\Lie_2(n)$ and the bases for $\sP_2(n)$ and show that Theorem \ref{main1} and Theorem \ref{main2} are equivalent to one another. 

\begin{prop}\label{com}
Fix the alphabet $X.$ 
Suppose for any subset $Y$ of $X,$ we have a basis $\sB(Y)$ for  $\Lie_2(|Y|)$ on the alphabet $Y$.  
We define $\sB_n^{\Com}(X)$ to be the set of products (under the commutative multiplication in $\sP_2(n)$) $b_{1} b_{2} \cdots b_{k},$ where each $b_i$ is in the basis $\sB(X_i)$ for $\Lie_2(|X_i|)$ (on the alphabet $X_i$),  and $\bigcup_{i=1}^k X_i$ is a partition of $X$ with $\max(X_1) < \cdots < \max(X_k).$ Then
$\sB_n^{\Com}(X)$ is a basis for $\sP_2(n).$
\end{prop}

The reason we define this set is natural: it is easy to prove by induction that each monomial in $\sP_2(n)$ can be written as a linear combinations of elements of the form of $m_1 m_2 \cdots m_k,$ where $m_i$ is a monomial in $\Lie_2(|X_i|)$ for each $i,$ and $\bigcup_{i=1}^k X_i$ is a partition of $X$ with $\max(X_1) < \cdots < \max(X_k).$ Therefore, we have the following proposition.
\begin{prop}
$\sB_n^{\Com}(X)$ spans $\sP_2(n).$
\end{prop}

Therefore, $\sB_n^{\Com}(X)$ is a basis candidate for $\sP_2(n).$ As usual, the proof of independence is more complicated. Proposition \ref{com} can be proved directly by more abstract methods; see Corollary 1 in \cite{dot-kho}. To make our paper self-contained, however, we include a different proof. We put the proof of independence in the next section, so that the uninterested reader can easily skip it.

Assuming Proposition \ref{com}, we can immediately construct a basis for $\sP_2(n)$ from $\cB_n(X),$ a basis for $\Lie_2(n).$
\begin{cor}\label{bncom}
Let $\cB_n^\Com(X)$ be the set of  products $b_{G_1} \cdots b_{G_k},$ where $G_1, \dots, G_k$ are components (or rooted trees) in the forest of rooted trees on $X$ with $\max(G_1) < \cdots < \max(G_k).$ Then $\cB_n^\Com(X)$ is a basis for $\sP_2(n).$
\end{cor}

Below is another corollary to Proposition \ref{com}.

\begin{cor}\label{equiv}
Theorem \ref{main1} is equivalent to Theorem \ref{main2}.
\end{cor}

\begin{proof}Let $l(n)$ and $p(n)$ be the ranks of $\Lie_2(n)$ and $\sP_2(n),$ respectively. By convention, we set $l(0) = 0$ and $p(0) =1.$ Define the exponential generating functions of $l(n)$ and $p(n)$ to be 
$$L(x) = \sum_{n=0}^\infty l(n) \frac{x^n}{n!}, \  P(x) = \sum_{n=0}^\infty p(n) \frac{x^n}{n!}.$$
By Proposition \ref{com}, we have
$$p(|X|) = \sum l(|X_1|) l(|X_2|) \cdots l(|X_k|),$$
where the sum is over all partitions $\bigcup_{i=1}^k X_i$ of $X$ with $\max(X_1) < \cdots < \max(X_k).$ 
Then by Corollary 5.1.6 of \cite{ec2}, we have $$P(x) = e^{L(x)}.$$
It is well known that if two exponential generating functions $L(x)$ and $P(x)$ satisfy the above formula, then $l(n) = n^{n-1}$ if and only if $p(n) = (n+1)^{n-1}.$ (See Section 5.3 of \cite{ec2} for a proof.)
\end{proof}

Since we already proved Theorem \ref{main1}, Theorem \ref{main2} follows from this corollary.

\section{Another perfect pairing and Quasi-binary trees}

The basic idea of the proof of the independence of $\sB_n^\Com(X)$ is the same as for the independence of $\cB_n(X):$ we use a perfect pairing. We first need to describe $\sP_2(n)$ in terms of combinatorial objects.

\begin{defn}
A {\it quasi-binary tree} is a rooted tree with root $r$ such that all leaves are odd-level vertices, and with the following orderings and degree restrictions:
\begin{alist}
\itm For any non-leaf odd-level vertex, it has degree two, and its children are ordered. In other words, we distinguish its left child and right child. If we switch the order of the left child and right child of an odd level vertex, we consider the newly obtained tree to be different from the original one.
\itm For any even-level vertex, it can have any nonzero degrees, and its children are not ordered. 
\end{alist}

Here, we use the same definition of the {\it level} of a vertex as in Definition \ref{l-level}/c): a vertex $x$ of $T$ is an $\ell$-level vertex if the unique path from $x$ to $r$ has length $\ell.$


A {\it $2$v-colored quasi-binary tree} is a quasi-binary tree whose odd-level vertices are colored by red or blue. We denote by $\QBT_n$ the set of all $2$v-colored quasi-binary trees whose leaves are labeled by $X.$

\end{defn}

\begin{rem}
We denote by $M_n^\Com$ the set of all the monomials in $\sP_2(n).$ Similarly to the case of $\BT_n$ and $M_n,$ there is a canonical bijection between $\QBT_n$ and $M_n^\Com:$ given a $2$v-colored quasi-binary tree, each leaf denotes a letter in $X,$ and we can construct a monomial in $M_n^\Com$ recursively by interpreting each odd-level vertex as a bracket of the left and right subtrees, with red vertices corresponding to $\brkone$ and blue vertices corresponding to $\brktwo,$ and interpreting each even-level vertex as a commutative product of its children.


\end{rem}

\begin{ex}[Example of the bijection between $\QBT_n$ and $M_n^\Com$]
Figure \ref{exmncom} shows the 2v-colored quasi-binary tree corresponding to the monomial $x_1 [x_2 x_3 x_4, \langle x_5, x_6 x_7 \rangle ].$ We use dashed lines to indicate edges below even-level vertices and solid lines to indicate edges below odd-level verrtices. Black vertices are even-level vertices.
\begin{figure}
\begin{center}
 
\input{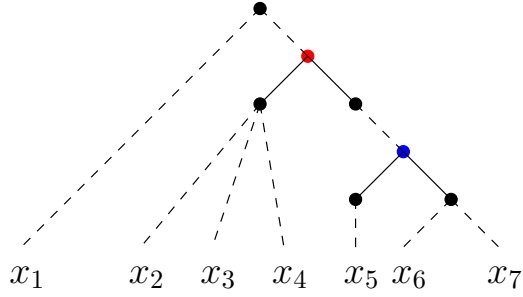}
 
\caption{The 2v-colored quasi-binary tree corresponds to the monomial $x_1 [x_2 x_3 x_4, \langle x_5, x_6 x_7 \rangle ].$}
\label{exmncom}
\end{center}
\end{figure}
\end{ex}

Since we can consider $\Lie_2(n)$ to be a submodule of $\sP_2(n),$ $BT_n = M_n$ is a subset of $M_n^\Com.$ Under the canonical bijection between $\QBT_n$ and $M_n^\Com,$ it is easy to see that $BT_n = M_n$ is in bijection with the set of trees $T$ in $\QBT_n$ satisfying
\begin{equation*}
\tag{$\diamond$} 
\mbox{each even-level vertex of $T$ has exact one child.}
\end{equation*}
In fact, for any tree $T$ in $\QBT_n$ satisfying ($\diamond$), if we contract all the edges below even-level vertices, then we obtain exactly the corresponding tree in $\BT_n.$ Therefore, we use the same notation $\BT_n$ to denote the set of trees $T$ in $\QBT_n$ satisfying ($\diamond$), and we can consider $\BT_n$ to be a subset of $\QBT_n.$ The graphs in Figure \ref{btqbt} show the 2v-colored binary tree and quasi-binary tree corresponding to the monomial $\langle [x_2, x_3], x_1 \rangle.$ It is easy to see one obtains the left graph by contracting all the edges below even-level vertices (or the dashed edges) in the right graph.

It is natural to extend the pairing we defined between $\BT_n$ and $\OG_n$ (Definition \ref{pairing}) to a pairing between $\QBT_n$ and $\OG_n.$

\begin{defn}\label{pairing2}
Given a $2$v-colored quasi-binary tree $T$ in $\QBT_n$ and an oriented two-colored graph $G$ in  $\OG_n,$ define 
$$\beta_{G,T}: \{ \mbox{edges of $G$} \} \to \{ \mbox{internal vertices of $T$} \}$$
by sending an edge $e: i \to j$ in $G$ to the nadir of the shortest path $p_T(e)$ from $i$ to $j$ on $T.$ Let $\tau_{G,T} = (-1)^N,$ where $N$ is the number of edges $e$ in $G$ for which $p_T(e)$ travles counterclockwise at its nadir. We say $\beta_{G,T}$ is {\it color-preserving} if for any edge $e \in G,$ $\beta_{G,T}(e)$ is an odd-level vertex of $T$ and the color of $\beta_{G,T}(e)$ is the same as the color of $e.$

Define the {pairing} of $G, T$ as
$$\llangle  G, T \rrangle_\Com =
\begin{cases}\tau_{G,T}, & \mbox{if $\beta_{G,T}$ is color-preserving, and gives a bijection between} \\
 & \mbox{\{edges of $G$\} and \{internal odd-level vertices of $T$\};} \\
0, & \mbox{otherwise.}
\end{cases}$$
\end{defn}

It is easy to check that this definition of $\beta_{G,T}$ is consistent with the one we defined in Definition \ref{pairing} when $T \in \BT_n.$ We immediately conclude the following lemma.

\begin{lem}\label{consist}
For any $T \in \BT_n$ and any $G \in \OG_n,$ we have
$$\llangle  G, T \rrangle_\Com = \llangle  G, T \rrangle.$$
\end{lem}

\begin{defn} 
Let $\Theta_n^\Com$ be the free $R$-module generated by the $2$v-colored quasi-binary trees in $\QBT_n$ and recall that $\Gamma_n$ is the free $R$-module generated by the oriented two-colored graphs in $\OG_n.$ Extend the pairing $\llangle \ , \ \rrangle_\Com$ defined in Definition \ref{pairing2} to one between $\Theta_n^\Com$ and $\Gamma_n$ by linearity.
\end{defn}

$\Theta_n^\Com$ is not isomorphic to $\sP_2(n),$ because we did not define relations between the elements of $\QBT_n.$ We now define a submodule of $\Theta_n^\Com$ which corresponds to the relations in $\sP_2(n).$ 
Besides symmetry combinations, Jacobi combinations, and mixed Jacobi combinations, the extra combinations we need to define are the ones corresponding to the derivation relations (D1) and (D2). 

We will not formally define symmetry combinations, Jacobi combinations, and mixed Jacobi combinations for $\Theta_n^\Com,$ because they are very similar to the combinations we defined for $\Theta_n.$ In fact, the pictures of symmetry combinations in $\Theta_n^\Com$ look exactly the same as those in $\Theta_n.$ The pictures of (mixed) Jacobi combinations look very similar.  In Figure \ref{qbtrel}, we show what a Jacobi combination corresponding to the relation (J2) in $\sP_2(n)$ looks like. 
Comparing with the (J2) in Figure \ref{btrel}, the only difference is that there is an extra (dashed) edge between the two involved colored vertices.

\begin{figure}
\begin{center}
 
\input{qbtrelshort.pstex_t}
 
\caption{Examples of elements that generate $J_n^\Com$}
\label{qbtrel}
\end{center}
\end{figure}

\begin{defn} \label{qbcomb}

A {\it derivation combination} in $\Theta_n^\Com$ has the form $T_1 - T_2 - T_3,$ where $T_1, T_2, T_3 \in \QBT_n$ satisfy the following: there exists an even-level vertex $v$ of $T_1,$ and another vertex $w$ which is the right child of a child of $v,$ such that we can divide the subtrees under $w$ into two groups, say $B$ and $C$ such that $T_2$ is obtained from $T_1$ by removing all the subtrees in group $C$ and connecting all of them under vertex $v,$ and $T_3$ is obtained from $T_1$ by removing all the subtrees in group $B$ and connecting all of them under vertex $v.$ We say a derivation combination is of type (D1) or (D2), depending on the color of the parent of $w.$ In Figure \ref{qbtrel}, we show what a derivation combination of type (D1)  looks like. 

Let $J_n^\Com \subset \Theta_n^\Com$ be the submodule generated by symmetry combinations, Jacobi combinations, mixed Jacobi combinations and derivation combinations. 
\end{defn} 

We now can describe $\sP_2(n)$ in terms of $\Theta_n^\Com$ and $J_n^\Com.$

\begin{lem}
$$\sP_2(n) \cong \Theta_n^\Com/J_n^\Com.$$
\end{lem}

\begin{prop}\label{vanish1com}
The pairing $\llangle \beta, \alpha \rrangle_\Com$ vanishes whenever $\alpha \in J_n^\Com.$
\end{prop}

We omit the proof of this Proposition, which can be shown analogously to the proof of Proposition \ref{vanish1}.

We next define a space $\sQ_2(n)$ corresponding to $\sP_2(n),$ and then show that $\llangle \ , \ \rrangle_\Com$ is a perfect paring between $\sP_2(n)$ and $\sQ_2(n).$
\begin{defn}
Let $I_n^\Com \subset \Gamma_n$ be the submodule generated by symmetry combinations, Jacobi combinations and mixed Jacobi combinations (defined in Definition \ref{ocomb}), as well as the graphs with more than one edge between two vertices. Let $$\sQ_2(n) := \Gamma_n/I_n^\Com.$$
\end{defn}

Note that $I_n^\Com$ is a submodule of $I_n.$ The difference between them is that $I_n$ contains disconnected graphs. Hence, $\Eil_2(n)$ is a submodule of $\sQ_2(n).$

We have the following proposition and lemma on $I_n^\Com$ and $\sQ_2(n).$ We omit the proofs of them, which are very similar to those of Proposition \ref{vanish2} and Lemma \ref{cyclic}.
\begin{prop}\label{vanish2com}
The pairing $\llangle \beta, \alpha \rrangle_\Com$ vanishes whenever $\beta \in I_n^\Com.$
\end{prop}

\begin{lem}\label{cycliccom}
For any oriented two-colored graph $G$, we have $G = 0$ in $\sQ_2(n)$ unless $G$ is a forest of trees on $X.$
\end{lem}

\begin{rem}
It is not true that $\llangle \beta, \alpha \rrangle_\Com$ vanishes whenever $\beta \in I_n,$ because when $\alpha = T \in \QBT_n \setminus \BT_n$ and $\beta = G$ is a disconnected graph in $\OG_n,$ the map $\beta_{G,T}$ could give a bijection between  \{edges of $G$\} and \{internal odd-level vertices of $T$\}. For example, if $T$ is the tree in Figure \ref{exmncom} and $G$ is the graph on vertex $\{x_1, \dots, x_7\}$ with a red edge $x_3 \to x_6$ and a blue edge $x_5 \to x_7,$ then $\llangle G, T\rrangle_\Com = 1.$

\end{rem}

By Propositions \ref{vanish1com} and \ref{vanish2com}, we can pass from the pairing  $\llangle , \rrangle_\Com$ between $\Theta_n^\Com$ and $\Gamma_n$ to a pairing between $\sP_2(n)$ and $\sQ_2(n).$ 
We still use the same notation  $\llangle , \rrangle_\Com$ to denote the pairing. 
\begin{prop}\label{com2}
Fix the alphabet $X.$ 
Suppose for any subset $Y$ of $X,$ we have a basis $\sO(Y)$ for $\Eil_2(|Y|)$ on the alphabet $Y$.
We define $\sO_n^{\Com}(X)$ to be the set of elements $o_{1} o_{2} \cdots o_{k}$, where each $o_i$ is in the basis $\sO(X_i)$ for $\Eil_2(|X_i|)$ (on the alphabet $X_i$),  and $\bigcup_{i=1}^k X_i$ is a partition of $X$ with $\max(X_1) < \cdots < \max(X_k).$ Then
$\sO_n^{\Com}(X)$ is a basis for $\sQ_2(n).$
\end{prop}
\begin{rem}
Note that each $o_{i} \neq 0 \in \Eil_2(|X_i|).$ Thus, by Lemma \ref{cyclic}, $o_{i}$ is a linear combination of trees on $X_i.$ Suppose for each $i,$ we have $o_i = \sum c_{i,j} G_{i,j}$ for some $c_{i,j}$ in $R,$ where $G_{i,j}$'s are trees on $X_i.$ In the definition of $\sO_n^{\Com}(X)$ in the proposition, by ``element $o_{1} o_{2} \cdots o_{k}$'',  we mean the element $$\sum_{j_1, \dots, j_k} (\prod_{i=1}^k c_{i, j_i}) \times (\mbox{ graph with } k \mbox{ components } G_{1, j_1}, \dots, G_{k, j_k})$$  in $\Gamma_n.$
\end{rem}

The proof of this proposition takes the remainder of this section. 

\begin{lem}\label{generate}
If $\sum_{j=1}^m a_j G_j = 0$ in $\Eil_2(n)$, and for each $j,$ $a_j \neq 0$ and $G_j$ is a tree on $X,$ then $\sum_{j=1}^m a_j G_j$ is generated by symmetry combinations, Jacobi combinations and mixed Jacobi combinations (defined in Definition \ref{ocomb}). 
\end{lem}

\begin{proof}

$\sum_{j=1}^m a_j G_j = 0$ implies that $\sum_{j=1}^m a_j G_j \in I_n$ is generated by five possible relations: symmetry combinations, Jacobi combinations and mixed Jacobi combinations, graphs with more than one edge between two vertices, and disconnected graphs. Therefore, we can find a sequence of elements in $I_n:$ 
$$\sum_{j=1}^m a_j G_j = \sum_{j=1}^{m_1} a_{1,j} G_{1,j} \to  \sum_{j=1}^{m_2} a_{2,j} G_{2,j} \to \cdots \to  \sum_{j=1}^{m_\ell} a_{\ell,j} G_{\ell,j} \to 0,$$ such that each element in the sequence is obtained by applying one of the five relations to the previous element and then possibly canceling out some graphs. We will prove the lemma by induction on $\ell.$

If $\ell = 1,$ then $\sum_{j=1}^m a_j G_j$ is one of the five relations (up to a scalar). Since all of $G_1, \dots, G_m$ are trees on $X,$ they cannot be graphs with more than one edge between two vertices or disconnected graphs. 

Assuming the proposition holds for $\ell < \ell_0,$ for some $\ell_0 \ge 2,$ we consider the case $\ell = \ell_0.$ Since all of $G_1, \dots, G_m$ are trees on $X,$ we can only apply a symmetry combination, a Jacobi combination or a mixed Jacobi combination on $\sum_{j=1}^m a_j G_j$ to obtain $\sum_{j=1}^{m_2} a_{2,j} G_{2,j}.$ One checks, in any of these three kinds of combinations, if one of the involved graphs is a tree, then the rest are trees as well. Therefore, all of the $G_{2,j}$'s are trees on $X.$ By the induction hypothesis, $\sum_{j=1}^{m_2} a_{2,j} G_{2,j}$ is  generated by symmetry combinations, Jacobi combinations and mixed Jacobi combinations. The desired result follows.

\end{proof}

\begin{prop}\label{sqspan} $\sO_n^\Com(X)$ spans $\sQ_2(n).$
\end{prop}

\begin{proof}
For any nonzero two-colored oriented graph $G \in \sQ_2(n),$ by Lemma \ref{cycliccom}, $G$ is a forest of trees on $X,$ that is, there exists  a partition of $X = \bigcup_{i=1}^k X_i$ with $\max(X_1) < \cdots < \max(X_k)$ such that $G$ has $k$ connected components $G_1, \dots, G_k,$ where $G_i \in \OG_{|X_i|}$ is an oriented two-colored tree on $X_i,$ for each $i.$

For each $i,$ because $\sO(X_i)$ is a basis for $\Eil_2(|X_i|),$ we can write $G_i$ as a linear combination of elements in $\sO(X_i)$. Hence, for some $a_{i,j} \in R,$ we have that $$G_i = \sum_{o_{i,j} \in \sO(X_i)} a_{i,j} o_{i,j} \ \mbox{ in $\Eil_2(|X_i|).$}$$  
Note that each 
$o_{i,j}$ is a linear combination of trees on $X_i.$ 
By Lemma \ref{generate}, $G_i - \sum a_{i,j} o_{i,j}$ is generated by symmetry combinations, Jacobi combinations and mixed Jacobi combinations (defined in Definition \ref{ocomb}).

Therefore, $$G - \sum_{j_1, \dots, j_k} (\prod_{i=1}^k a_{i, j_i}) \times (\mbox{ element } o_{1, j_1}, \dots, o_{k, j_k} )$$
is generated by symmetry combinations, Jacobi combinations and mixed Jacobi combinations, and thus is an element in $I_n^\Com.$ Hence, $G$ can be written as a linear combination of 
elements 
 $o_{1} o_{2} \cdots o_{k}$ in $\sQ_2(n),$ where each $o_i$ is in the basis $\sO(X_i).$
\end{proof}

Now we have basis candidates for both $\sP_2(n)$ and $\sQ_2(n).$ So it is enough to show that the matrix of the pairing $\llangle , \rrangle_\Com$ between $\sB^\Com(X)$ and $\sO^\Com(X)$ is nonsingular.

\begin{lem}\label{vanish3com}
Suppose $\alpha \in \sB^\Com(X)$ is the element in $\Theta_n^\Com$ corresponding to a product $b_{1} b_{2} \cdots b_{k},$ where $\bigcup_{i=1}^k X_i$ is a partition of $X$ and each $b_i$ is in the basis $\sB(X_i)$ for $\Lie_2(|X_i|),$ and  $\beta \in \sO^\Com(X)$ is an element $o_{1} o_{2} \cdots o_{k'},$ where $\bigcup_{i=1}^{k'} X_{i}'$ is a partition of $X$ and each $o_i$ is in the basis $\sO(X_i').$

 $\llangle  \beta, \alpha \rrangle_\Com = 0$ unless $k = k',$ and $\bigcup_{i=1}^k X_i$ and $\bigcup_{i=1}^{k'} X_{i}'$ are the same partition. 
 

\end{lem} 
\begin{rem}
For any monomials $m_1, \dots, m_k$ in $M_n^\Com,$ suppose the 2v-colored quasi-binary tree corresponding to $m_i$ is $T_i,$ for each $i.$ Then the 2v-colored quasi-binary tree corresponds to the product $m_1 m_2 \cdots m_k$ is the tree obtained by gluing the roots of all the $T_i$'s together. 
\end{rem}
\begin{proof}
Suppose $b_i = \sum c_{i,j} m_{i,j},$ where each $m_{i,j}$ is a monomial in $\Lie_2(|X_i|)$, and $o_i = \sum c_{i,j}' G_{i,j},$ where each $G_{i,j}$ is a tree on $X_i'.$ We denote by $G_{j_1, \dots, j_{k'}}$ the graph with $k'$ components $G_{1,j_1}, \dots, G_{k', j_{k'}}$. Then
\begin{eqnarray*}
& &\llangle  G, T \rrangle_\Com \\
&=& \llangle  \sum_{j_1, \dots, j_{k'}} (\prod_{i=1}^{k'} c_{i, j_i}') \times G_{j_1, \dots, j_{k'}}\ , \ \sum_{j_1', \dots, j_k'} (\prod_{i=1}^k c_{i, j_i'}) \times (m_{1,j_1'} \cdots m_{k, j_k'}) \rrangle_\Com \\
&=&  \sum_{j_1, \dots, j_{k'}} \sum_{j_1', \dots, j_k'} (\prod_{i=1}^{k'} c_{i, j_i}')(\prod_{i=1}^k c_{i, j_i'}) \llangle  G_{j_1, \dots, j_{k'}} \ , \   m_{1,j_1'} \cdots m_{k, j_k'} \rrangle_\Com
\end{eqnarray*}

Note that $\llangle  \beta, \alpha \rrangle_\Com \neq 0$ implies that one of the $\llangle  G_{j_1, \dots, j_{k'}} \ , \   m_{1,j_1'} \cdots m_{k, j_k'} \rrangle_\Com$ is not zero. Therefore, it is sufficient to check the case when $\alpha = T$ is the tree in $\QBT_n$ corresponding to a product $b_{1} b_{2} \cdots b_{k},$ where each $b_i$ is a monomial in $\Lie_2(|X_i|),$ and $\beta = G \in \OG_n$ is a graph with $k'$ components $o_1 \cdots o_{k'},$ where each $o_i$ is a tree on $X_i'.$

If $\llangle  G, T \rrangle_\Com \neq 0$, then $\beta_{G,T}$ gives a bijection between  \{edges of $G$\} and \{internal odd-level vertices of $T$\}.
Since each $o_i$ is a tree on $X_i',$ the number of edges in $G$ is $n-k'.$  Thus, the number of internal odd-level vertices of $T$ is $n-k'.$ However, for $T,$ we have
$$n - 1 =  \sum_{v: \mbox{ an internal vertex of $T$}} (-1 + \# \mbox{ children of } v).$$

Each internal odd-level vertex of $T$ has exactly two children. Because each $b_i$ is a monomial in $\Lie_2(|X_i|),$ all the even-level vertices of the 2v-colored quasi-binary tree corresponding to $b_i$ has exactly one child. Therefore, any non-root even-level vertex of $T$ has exactly one child. Hence, we have
$$n - 1 = (\# \mbox{ internal odd-level vertices of $T$} ) + (-1 + \# \mbox{ children of the root of $T$} ).$$
The number of children of the root of $T$ is $k.$ Therefore, the number of internal odd-level vertices of $T$ is $n-k.$ So $k = k'.$

For any $x, y \in X,$ if $x \in X_i$ and $y \in X_j$ with $i \neq j,$ then $e = x \to y$ or $y \to x$ is not an edge in $G,$ because otherwise the nadir of $p_T(e)$ is the root of $T.$ Hence, $e$ is an edge of $G$ only when the two ends of $e$ are in the same set $X_i,$ for some $i.$ Therefore, each $X_i'$ has to be a subset of $X_{j_i}$ for some $j_i.$ Given $k = k',$ we must have that $\bigcup_{i=1}^k X_i$ and $\bigcup_{i=1}^{k'} X_{i}'$ are the same partition. 

\end{proof}

Lemma \ref{vanish3com} implies that if we choose a proper order, the matrix of the pairing $\llangle , \rrangle_\Com$ between $\sB^\Com(X)$ and $\sO^\Com(X)$ is a block diagonal matrix, where the blocks on the diagonal correspond to all the partitions $\bigcup_{i=1}^k X_i$ of $X$ with $\max(X_1) < \cdots < \max(X_k).$

\begin{prop}\label{nonsingcom}
Suppose $par := \bigcup_{i=1}^k X_i$ is a partition of $X$ with $\max(X_1) < \cdots < \max(X_k),$ let $\sB_{par}^\Com(X)$ and $\sO_{par}^\Com(X)$ be the subsets of $\sB^\Com(X)$ and  $\sO^\Com(X)$ respectively, corresponding to this partition. (Note that it is easy to verify that $\sB_{par}^\Com(X)$ and $\sO_{par}^\Com(X)$ have the same cardinality.) Then the matrix of the pairing $\llangle , \rrangle_\Com$ between $\sB_{par}^\Com(X)$ and $\sO_{par}^\Com(X)$ is nonsingular.

\end{prop}

Before we prove Proposition \ref{nonsingcom}, we first review some basic results on {\it Kronecker products of matrices} \cite[Section 4.2]{horn-johnson}, which we will need in the proof. Recall the {\it Kronecker product} of an $n \times n$ matrix $A = (a_{i,j})$ and an $m \times m$ matrix $B$ is the $mn \times mn$ matrix
$$A \otimes B = \left(\begin{array}{ccc}a_{1,1} B & \cdots  & a_{1,n} B \\
\vdots & \ddots & \vdots \\
a_{n,1}B & \cdots & a_{n,n} B\end{array}\right).$$
The Kronecker product is bilinear and associative. We have the following lemma:
\begin{lem}\cite[Corollay 4.2.11]{horn-johnson}\label{kron}
If $A$ and $B$ are both nonsingular, then so is $A \otimes B$. 
\end{lem}

\begin{proof}[Proof of Proposition \ref{nonsingcom}]
It is enough to give certain orders on the elements of $\sB_{par}^\Com(X)$ and $\sO_{par}^\Com(X)$, and show the matrix of the pairing $\llangle , \rrangle_\Com$ according to the given orders is nonsingular.

For each $X_i,$ we already know that the cardinality of both $\sB(X_i)$ and $\sO(X_i)$ is  $|X_i|^{|X_i|-1}$ (Theorem \ref{main1} and Corollary \ref{eilrank}). For simplicity, we let $j_i := |X_i|^{|X_i|-1}.$ We fix an order for elements in $\sB(X_i) = \{ b_{i, 1} <  \cdots < b_{i, j_i} \}$ and an order for the elements in $\sO(X_i) = \{ o_{i,1} < \cdots < o_{i, j_i} \}.$ Let $\cM_i$ be the matrix of the pairing $\llangle , \rrangle_\Com$ between $\sB(X_i)$ and $\sO(X_i)$ according to the fixed ordering, i.e., the $(\ell, m)$-entry of $\cM_i$ is given by $\llangle o_{i, \ell}, b_{i, m} \rrangle_\Com.$ By Lemma \ref{consist}, $\cM_i$ is in fact the the matrix of the pairing $\llangle , \rrangle$ between $\sB(X_i)$ and $\sO(X_i).$ Because $\llangle , \rrangle$ is a perfect pairing between $\Lie_2(|X_i|)$ and $\Eil_2(|X_i)$ (Theorem \ref{main3}), the matrix $\cM_i$ is nonsingular.

We give a lexicographic order on the elements of $\sB_{par}^\Com(X)$ according to the orders we fixed on $\sB(X_i)$'s: for any two distinct elements $b_1 \cdots b_k$ and $b_1' \cdots b_k'$ in $\sB_{par}^\Com(X),$ where $b_i, b_i' \in \sB(X_i)$ for each $i,$ we say $b_1 \cdots b_k < b_1' \cdots b_k'$ in $\sB_{par}^\Com(X)$ if at the first position, say $\ell,$ these two elements differ, we have $b_\ell < b_\ell'$ in $\sB(X_i).$ Hence, the order of the elements in $\sB_{par}^\Com(X)$ looks like:
\begin{eqnarray*}
& &b_{1,1} \cdots b_{k-1, 1} b_{k,1} \ < \ b_{1,1}  \cdots b_{k-1, 1} b_{k,2} \ < \ \cdots \ < \ b_{1,1} \cdots b_{k-1, 1} b_{k,j_k} \\
&<&  b_{1,1}  \cdots b_{k-2, 1} b_{k-1, 2} b_{k,1} \ < \ b_{1,1} \cdots b_{k-2, 1} b_{k-1, 2} b_{k,2} \ < \ \cdots \ < \ b_{1,1} \cdots b_{k-2, 1} b_{k-1, 2} b_{k,j_k}  \\
& & \cdots \cdots \cdots \cdots \cdots \\
&<& b_{1,j_1} \cdots b_{k-1, j_{k-1}} b_{k,1} \ < \ b_{1,j_1} \cdots b_{k-1, j_{k-1}} b_{k,2} \ < \  \cdots \ < \ b_{1,j_1} \cdots b_{k-1, j_{k-1}} b_{k,j_k}. 
\end{eqnarray*}
Similarly, we give a lexicographic order on the elements of $\sO_{par}^\Com(X)$ according to the orders we fixed on $\sO(X_i)$'s. Let $\cM$ be the matrix of the pairing $\llangle , \rrangle_\Com$ between $\sB_{par}^\Com(X)$ and $\sO_{par}^\Com(X)$ according the two orders we just defined. One can check that $\cM$ is the Kronecker products of $\cM_1, \dots, \cM_k:$
$$\cM = \cM_1 \otimes \cdots \otimes \cM_k.$$
Since all of $\cM_1, \dots, \cM_k$ are nonsingular, by using Lemma \ref{kron} $k-1$ times, we conclude that $\cM$ is nonsingular.
\end{proof}

Proposition \ref{com} and Proposition \ref{com2} follows from Proposition \ref{nonsingcom} and Lemma \ref{vanish3com}. We can also conclue:

\begin{thm}\label{main5}
The pairing $\llangle , \rrangle_\Com$ between $\sP_2(n)$ and $\sQ_2(n)$ is perfect.
\end{thm}


%

\section{Further discussion and questions}\label{questions}

One notices that for all the relations (S1), (S2), (J1), (J2) and (MJ) we have in $\Lie_2(n),$ the elements in each of them has exactly the same number of $\brkone$'s and $\brktwo$'s. Therefore, it is natural to consider the following submodules of $\Lie_2(n):$

\begin{defn}
For any $i = 0, 1, \dots, n-1,$ we define $\Lie_2(n,i)$ to be the submodule of $\Lie_2(n)$ that is generated by all the monomials in $\Lie_2(n)$ with exactly $i$ $\brkone$'s (and $n-1-i$ $\brktwo$'s). 
\end{defn}

It is clear that we can write $\Lie_2(n)$ as the direct sum of $n$ submodules:

\begin{lem}
$$\Lie_2(n) = \bigoplus_{i=0}^{n-1} \Lie_2(n,i)$$ 
\end{lem}

$\cB_n(X) = \{ b_G \ | \ G \in \overline{\cG}_n = \cR_n \}$ is a basis for $\Lie_2(n),$ and for any $G \in \overline{\cG}_n,$ the number of $\brkone$'s in $b_G$ is equal to the number of red edges in $G$, or equivalently, the number of increasing edges in $G$ when considering $G$ as a rooted tree. Thus, we obtain the bases for $\Lie_2(n,i)$'s.

\begin{prop}\label{subrank}
The set $\cB_{n,i}(X) := \{ b_G \ | \ G \in \overline{\cG}_n \mbox{ has $i$ red edges} \} =  \{ b_G \ | \ G \in \cR_n \mbox{ has $i$ increasing edges} \}$ is a basis for $\Lie_2(n,i).$

Hence, the rank of $\Lie_2(n,i)$ equals to the number of rooted trees on $n$ vertices with $i$ increasing edges.
\end{prop}

Noting that $\Lie(n) \cong \Lie_2(n,n-1),$ we recover the formulas for the rank of $\Lie(n).$

\begin{cor}
$\Lie(n)$ is free of rank $(n-1)!.$
\end{cor}
\begin{proof}
The rank of $\Lie(n)$ equals to the rank of $\Lie_2(n,n-1)$. By Proposition \ref{subrank}, the rank is the number of increasing trees on $n$ vertices. (Here by increasing trees, we mean rooted trees with all the edges are increasing edges.) However, it is well known \cite[page 82]{kuz-pak-pos} that the number of increasing trees on $n$ vertices is $(n-1)!.$
\end{proof}


If we denote by $a(n,i)$ the number of rooted trees on $n$ vertices with $i$ increasing edges, then by the exponential generating function for the $SL_2$-characters for $\Lie_2(n)$ with $SL_2$ action obtained in \cite{dot-kho}, we get the generating function for $a(n,i).$

\begin{cor}
\begin{equation}\label{geninc}
\sum_{i=0}^{n-1} a(n,i) x^i = \prod_{k=1}^{n-1} (k x + (n-k)).
\end{equation}
Hence, the number of rooted trees on $n$ vertices with $i$ increasing edges is given by
\begin{equation}\label{inc}
a(n,i) = \sum_{K: \mbox{ a $i$-subset of $[n-1]$}} \prod_{k \in K} k \prod_{k' \in [n-1] \setminus K} (n-k').
\end{equation}
\end{cor}
\begin{proof}
By formula (16) in \cite{dot-kho}, we have
$$\sum_{i=0}^{n-1} a(n,i) q^{n-1-2i} = \sum_{k=1}^{n-1} (kq + (n-k)q^{-1}).$$
We can obtain \eqref{geninc} by multiplying $q^{n-1}$ on both sides of the above formula, setting $x = q^2,$ reindexing the left side, and then applying the fact $a(n,i) = a(n, n-1-i).$
\end{proof}
We ask the following question:
\begin{ques}
Can one find a combinatorial proof for formulas \ref{geninc} and \ref{inc}?
\end{ques}

As we see, $\Lie(n)$ is a submodule of $\Lie_2(n).$ Thus, we can consider $\Lie_2(n)$ to be a generalization of $\Lie(n).$ Hence, another question which might be interesting is:
\begin{ques}
Can we generalize $\Lie(n)$ further? Is it possible to define $\Lie_k(n)$ for any $k \ge 1$ so that it has nice rank formulas like those for $\Lie(n)$ and $\Lie_2(n)$? What are the right combinatorial objects for $\Lie_k(n)$, if it can be defined?
\end{ques}

\bibliographystyle{amsplain}
\bibliography{gen}

\end{document}